\documentclass[a4paper,11pt,reqno]{amsart}
\usepackage[latin1]{inputenc}
\usepackage[T1]{fontenc}
\usepackage[english]{babel}
\usepackage{amsfonts, amsmath, amssymb}
\usepackage{dsfont} 
\usepackage{braket}
\usepackage[all,cmtip]{xy}
\usepackage{color}

\usepackage{graphicx}
\graphicspath{{graphics/}}
\usepackage[colorlinks=true,urlcolor=blue,citecolor=black,linkcolor=black]
{hyperref}
\newcommand{\C}{\mathbb{C}}
\newcommand{\R}{\mathbb{R}}
\newcommand{\M}{\mathrm{M}}
\newcommand{\cA}{\mathcal{A}}
\newcommand{\cB}{\mathcal{B}}
\newcommand{\cC}{\mathcal{C}}
\newcommand{\cD}{\mathcal{D}}
\newcommand{\cE}{\mathcal{E}}
\newcommand{\cF}{\mathcal{F}}
\newcommand{\cG}{\mathcal{G}}
\newcommand{\cH}{\mathcal{H}}
\newcommand{\cK}{\mathcal{K}}
\newcommand{\cN}{\mathcal{N}}
\newcommand{\cO}{\mathcal{O}}
\newcommand{\cR}{\mathcal{R}}
\newcommand{\cS}{\mathcal{S}}
\newcommand{\fg}{\mathfrak{g}}
\newcommand{\ii}{\mathrm{i}}
\newcommand{\id}{\mathds{1}}
\renewcommand\Im{\mathrm{Im}}
\renewcommand\Re{\mathrm{Re}}
\newcommand{\tp}{{}^\mathrm{T}}
\newcommand{\dcs}{\oplus_\mathrm{c}}
\DeclareMathOperator\aff{aff}
\DeclareMathOperator\argmax{argmax}
\DeclareMathOperator\conv{conv}
\DeclareMathOperator\diag{diag}
\DeclareMathOperator\ex{ex}
\DeclareMathOperator\idty{id}
\DeclareMathOperator\Img{Img}
\DeclareMathOperator\red{red}
\DeclareMathOperator\ri{ri}
\DeclareMathOperator\spn{span}
\DeclareMathOperator\tr{tr}
 \newtheorem{thm}{Theorem}[section]
 \newtheorem{cor}[thm]{Corollary}
 \newtheorem{lem}[thm]{Lemma}
 \newtheorem{prop}[thm]{Proposition}
 \theoremstyle{definition}
 \theoremstyle{remark}
 \newtheorem{rem}[thm]{Remark}
 \newtheorem{exa}[thm]{Example}
 \numberwithin{equation}{section}
\begin{document}
\selectlanguage{English}
\title{Matrix systems, algebras, and open maps}
\author{Stephan Weis}
\begin{abstract}
Every state on the algebra $\M_n$ of complex $n\times n$ matrices restricts 
to a state on any matrix system. Whereas the restriction to a matrix system 
is generally not open, we prove that the restriction to every *-subalgebra 
of $\M_n$ is open. This simplifies topology problems in matrix theory and 
quantum information theory.
\end{abstract}
\date{Draft of \today}
%
%
%
\maketitle
%
%
%
%
\hfill
In honor of Ilya Matveevich Spitkovsky, for his 70th birthday.
\par
%
%
%
\section{Introduction}
\label{sec:intro}
In the work of Choi and Effros~\cite{ChoiEffros1977}, a \emph{matrix system} 
on $\C^n$ is a complex linear subspace $\cR$ of the full matrix algebra $\M_n$ 
that is self-adjoint (the conjugate transpose $A^\ast$ of every $A\in\cR$ lies 
in $\cR$) and contains the $n\times n$ identity matrix $\id_n$, see also 
\cite{Arveson2010,Paulsen2002}. Let $\cR$ be a matrix system on $\C^n$. If 
$\cR$ is closed under matrix multiplication we call it a 
\emph{*-subalgebra} of $\M_n$. The dual space to $\cR$ is denoted by 
$\cR^\ast:=\{\ell:\cR\to\C\mid \text{$\ell$ is $\C$-linear}\}$ and the cone 
of positive semidefinite matrices in $\cR$ by $\cC(\cR)$. The 
\emph{state space} of $\cR$ is
\begin{align*}
\cS(\cR) & := 
\{\ell\in\cR^\ast \mid \forall A\in\cC(\cR):\ell(A)\geq 0, \ell(\id_n)=1 \}\,.
\end{align*}
\par
The restriction $\cS(\M_n)\to\cS(\cR)$, $\ell\mapsto\ell|_{\cR}$ of states to 
$\cR$ is continuous and affine. Its analytic properties would be perfectly 
clear if it were not for the openness that fails in Example~\ref{ex:ex1}. Let 
$K,L$ be subsets of some Euclidean spaces endowed with their relative topologies 
\cite{Kelley1975}. A map $f:K\to L$ is \emph{open at} $x\in K$ if the image 
of every neighborhood of $x$ in $K$ is a neighborhood of $f(x)$ in $L$. The 
map $f$ is \emph{open} if it is open at every point in $K$. 
\par
It is helpful to represent states as matrices. The antilinear isomorphism 
$\cR\to\cR^\ast$, $A\mapsto\langle A,\,\cdot\,\rangle$ restricts by 
Lemma~\ref{lem:riesz} to the affine isomorphism
\[\textstyle
r_\cR:\cD(\cR)\to\cS(\cR)\,,
\quad 
\rho\mapsto\langle \rho,\,\cdot\,\rangle\,,
\]
where $\langle A,B\rangle:=\tr(A^\ast B)$ is the Frobenius inner 
product of $A,B\in\M_n$,
\begin{align*}
\cH(\cR) 
& :=\{A\in\cR\mid A^\ast=A\}\,,\\ 
\cC(\cR)^\vee 
& := \{A\in\cH(\cR) \mid \forall B\in\cC(\cR):\langle A,B\rangle\geq 0\}\,,\\
\text{and}\quad
\cD(\cR) 
& := \{\rho\in\cC(\cR)^\vee \mid \tr(\rho)=1 \}\,.
\end{align*}
Generalizing a term of von Neumann algebras \cite{BratteliRobinson1987},
we refer to the elements of $\cD(\cR)$ as \emph{density matrices}. The inclusion 
$\cC(\cR)\subset\cC(\cR)^\vee$ can be strict. For example, the density matrix 
$\diag(-1,5,2)/6$ of the matrix system spanned by $\id_3$ and $\diag(1,-1,0)$ is 
indefinite. It is well known, see Rem.~\ref{rem:intersection}, that the identity 
$\cC(\cA)=\cC(\cA)^\vee$ holds for every *-subalgebra $\cA$ of $\M_n$.
\par
\begin{exa}\label{ex:ex1}
We write block diagonal matrices as direct sums, for instance
\[\textstyle
\begin{pmatrix}
A & 0\\
0 & c
\end{pmatrix}
=A\oplus c\in\M_3
\quad
\text{for every}
\quad
A\in\M_2\,,
\quad
c\in\C\cong\M_1\,.
\]
Denoting the imaginary unit by $\ii\in\C$ and the \emph{Pauli matrices} by
\[\textstyle
X:=\begin{pmatrix}
0&1\\1&0
\end{pmatrix}\,,
\quad
Y:=\begin{pmatrix}
0&-\ii\\\ii&0
\end{pmatrix}\,,
\quad 
Z:=\begin{pmatrix}
1&0\\0&-1
\end{pmatrix}\,,
\]
we define the matrix system $\cR:=\spn_\C(\id_3,X\oplus 1,Z\oplus 0)$.
The orthogonal projection of $\C^2$ onto the line spanned by 
$\ket{+}:=(1,1)\tp/\sqrt{2}$ is written $\ket{+}\!\!\bra{+}$ in 
\emph{Dirac's notation} \cite{BengtssonZyczkowski2017,Buckley2021}. 
The open set $\cO:=\{\ell\in\cS(\M_3)\mid\ell(0\oplus 1)>0\}$
contains 
\[\textstyle
\omega_\lambda
:=r_{\M_3}\big[(1-\lambda)\ket{+}\!\!\bra{+}\oplus\lambda\big]\,,
\quad
\lambda\in(0,1]\,.
\]
So $\cO|_\cR:=\{\ell|_\cR:\ell\in\cO\}$ contains $\omega_\lambda|_\cR$
but none of the restrictions $\ell_\theta|_\cR$ of
\[\textstyle
\ell_\theta:=
r_{\M_3}\big[\frac{1}{2}\big(\id_2+\cos(\theta)X+\sin(\theta)Z\big)\oplus 0\big]\,,
\quad
\theta\in(0,2\pi)\,.
\]
Specifically, $\ell_\theta$ has the value $\ell_\theta(A_\theta)=1$ at 
$A_\theta:=\cos(\theta)(X\oplus 1)+\sin(\theta)Z\oplus 0$ and
$\ell(A_\theta)\leq (\cos(\theta)-1)\ell(0\oplus 1)+1$ holds for every 
$\ell\in\cS(\M_3)$. Hence
\[\textstyle
\ell|_\cR(A_\theta)-\ell_\theta|_\cR(A_\theta)
\leq(\cos(\theta)-1)\ell(0\oplus 1)<0\,,
\quad
\ell\in\cO\,.
\]
This shows that $\cO|_\cR$ is not a neighborhood of $\omega_\lambda|_\cR$
as $\lim_{\theta\to0}\ell_\theta|_\cR=\omega_\lambda|_\cR$. In conclusion, 
$\cS(\M_3)\to\cS(\cR)$, $\ell\mapsto\ell|_{\cR}$ is not open at 
$\omega_\lambda$ for any $\lambda\in(0,1]$.
\end{exa}
Asking where $\cS(\M_n)\to\cS(\cR)$, $\ell\mapsto\ell|_{\cR}$ is open is the
same, by Coro.~\ref{cor:restriction}~b), as inquiring at which density 
matrices the orthogonal projection
\begin{equation}\label{eq:projection}
\cD(\M_n)\to\cD(\cR)
\end{equation}
is open. The map \eqref{eq:projection} is defined in \eqref{eq:pi-red1} and 
\eqref{eq:pi-red2} below. For now suffice it to say that it is a restriction 
of the orthogonal projection of $\M_n$ onto $\cR$. 
\par
Corey et al.~\cite{Corey-etal2013} and Leake et 
al.~\cite{Leake-etal2014-preimages,Leake-etal2014-inverse} first studied a 
problem of numerical ranges closely related to \eqref{eq:projection}. The 
problem \eqref{eq:projection} itself was studied by Weis \cite{Weis2014,Weis2016} 
and Rodman et al.~\cite{Rodman-etal2016} when $\cD(\cR)$ is replaced with 
the affinely isomorphic joint numerical range. Numerical ranges are the 
topic of Sec.~\ref{sec:numerical-ranges}. 
\par
\begin{thm}\label{thm:Aopen}
If $\cA$ is a *-subalgebra of $\M_n$, then the orthogonal projection 
$\cD(\M_n)\to\cD(\cA)$ is open.
\end{thm}
Thm.~\ref{thm:Aopen} can simplify the problem \eqref{eq:projection}, and 
related continuity problems, if $\cR$ is included in a *-subalgebra of 
$\M_n$ smaller than $\M_n$ as we show in Sec.~\ref{sec:simpler-subalgebras}. 
Examples from quantum information theory are presented in 
Sec.~\ref{sec:quantum}.
\par
Thm.~\ref{thm:Aopen} is proved in Sec.~\ref{sec:proof-thm}. The main ideas 
are that $\cA$ is a direct sum of full matrix algebras and that $\cD(\M_n)$ 
is stable and highly symmetric. Thereby, a convex set $K$ is \emph{stable} 
if the \emph{midpoint map}
\[\textstyle
K\times K\to K\,,
\quad
(x,y)\mapsto\frac{1}{2}(x+y)
\]
is open. A convex set is always (except Rem.~\ref{rem:intro}) understood to 
be included in a Euclidean space. Debs \cite{Debs1979} proved the stability 
of $\cD(\cA)$ for $\cA=\M_n$, Papadopoulou \cite{Papadopoulou1982} achieved 
it for *-subalgebras $\cA$ of $\M_n$, and we do it for \emph{real} 
*-subalgebras $\cA$ of $\M_n$ in Sec.~\ref{sec:retraction} using Clausing's 
work on retractions \cite{Clausing1978}. The analogue of Thm.~\ref{thm:Aopen} 
is established for the algebra $\cA=\M_n(\R)$ of real $n\times n$ matrices 
in Sec.~\ref{sec:real}. The Secs.~\ref{sec:states} and~\ref{sec:directsums} 
collect preliminaries.
\par
\begin{rem}\label{rem:intro}~
\begin{enumerate}
\item[a)]
Vesterstrøm \cite{Vesterstrom1973} proved that the restriction of states 
to the center of a von Neumann algebra is open. Thm.~\ref{thm:Aopen}
is a noncommutative analogue in the finite-dimensional setting.
\item[b)]
Stability is a meaningful concept in optimal control 
\cite{Papadopoulou1982,Witsenhausen1968} and quantum information theory 
\cite{Shirokov2012} because of the following ``CE-property''. A 
compact convex set $K$ is stable if and only if for every continuous 
function $f:K\to\R$, the envelope $f^\vee(x):=\sup\{h(x): h\leq f\}$, 
$x\in K$ is continuous, see \cite{Vesterstrom1973,Brien1976} and
\cite{Shirokov2012}. Here, the supremum is taken over all continuous 
affine functions $h:K\to\R$ whose graphs lie below the graph of $f$. 
\end{enumerate}
\end{rem}
%
%
%
\section{States and density matrices}
\label{sec:states}
The aim of this section is to translate openness questions from states to 
density matrices. 
In the sequel, we refer to matrix systems and *-subalgebras synonymously as 
\emph{complex matrix systems} and \emph{complex *-subalgebras}, respectively.
In addition, we introduce real *-subalgebras as they have a greater variety 
of state spaces than the complex ones. See, however, 
Rem.~\ref{rem:realvscomplex}.
\par
A \emph{real matrix system} on $\C^n$ is a real linear subspace $\cR$ of $\M_n$ 
that is self-adjoint and contains $\id_n$. Let $\cR$ denote a real matrix 
system on $\C^n$. We endow $\cR$ with the Euclidean scalar product 
\begin{equation}\label{eq:scalarp}
\cR\times\cR\to\R\,,
\quad
(A,B)\mapsto\Re\langle A,B\rangle\,,
\end{equation} 
where $\langle A,B\rangle=\tr A^\ast B$ and $\Re(a+\ii b)=a$ is the real part of 
a complex number, $A,B\in\M_n$, $a,b\in\R$. The positive cone, space of hermitian 
matrices, dual cone, and space of density matrices are defined verbatim to their 
respective complex counterparts defined in Sec.~\ref{sec:intro}, and are denoted 
by 
\[
\cC(\cR),
\quad 
\cH(\cR),
\quad 
\cC(\cR)^\vee,
\quad\text{and}\quad 
\cD(\cR)\,.
\]
We call $\cR$ a \emph{real *-subalgebra} of $\M_n$ if $\cR$ is closed under 
matrix multiplication. 
\par
Generalizing a definition from real algebras \cite[Sec.~4.5]{Li2003}, 
we define the \emph{real state space} of $\cR$ as
\[
\cS_\R(\cR):=\{\ell\in\cR_{\R,0}^\ast\mid 
\forall A\in\cC(\cR):\ell(A)\geq 0, \ell(\id_n)=1\}\,,
\]
where
\[
\cR_{\R,0}^\ast:=
\{\ell:\cR\to\R \mid \mbox{$\ell$ is $\R$-linear},
\forall A\in\cH^-(\cR):\ell(A)=0\}
\]
is the space of real functionals vanishing on the skew-hermitian 
matrices
\[
\cH^-(\cR):=\{A\in\cR:A^\ast=-A\}\,.
\]
\par
\begin{lem}\label{lem:riesz-real}
If $\cR$ is a real matrix system on $\C^n$, then the map
\[
r_{\R,\cR}:\cD(\cR)\to\cS_\R(\cR), 
\quad 
\rho\mapsto\Re\langle\rho,\,\cdot\,\rangle
\]
is a real affine isomorphism between compact convex sets.
\end{lem}
\begin{proof}
As $\cR$ is the orthogonal direct sum $\cR=\cH(\cR)\oplus\cH^-(\cR)$,
the real linear isomorphism \cite[Sec.~67]{Halmos1974}
\[
r_{\R,\cR}:
\cR\to\{\ell:\cR\to\R\mid \mbox{$\ell$ is $\R$-linear} \}\,,
\quad
A\mapsto\Re\langle A,\,\cdot\,\rangle
\]
restricts to the real linear isomorphism $\cH(\cR)\to\cR_{\R,0}^\ast$. 
Restricting $r_{\R,\cR}$ further, we obtain an injective map whose 
domain is $\cD(\cR)$. So, it suffices to prove 
$r_{\R,\cR}(\rho)\in\cS_\R(\cR)$ for 
$\rho\in\cD(\cR)$, and that $r_{\R,\cR}:\cD(\cR)\to\cS_\R(\cR)$ is 
surjective. Both assertions are straightforward to verify.
The convex set $\cD(\cR)$ is compact, see \cite[Lemma~3.3]{Plaumann-etal2021},
because $\id_n$ lies in the interior of $\cC(\cR)$ in the topology of 
$\cH(\cR)$. The convex set $\cS_\R(\cR)$ is compact as it is the image of a
compact set under a continuous map.
\end{proof}
Returning to complex functionals, we consider the real vector space
\[
\cR_\mathrm{her}^\ast:=
\{\ell:\cR\to\C \mid \mbox{$\ell$ is $\C$-linear},
\forall A\in\cH(\cR):\ell(A)\in\R\}
\]
of complex functionals taking real values on hermitian matrices.
\par
\begin{lem}\label{lem:riesz}
If $\cR$ is a complex matrix system on $\C^n$, then the map
\[
r_\cR:\cD(\cR)\to\cS(\cR)\,,
\quad 
\rho\mapsto\langle\rho,\,\cdot\,\rangle
\]
is a real affine isomorphism between compact convex sets.
\end{lem}
\begin{proof}
The complex antilinear isomorphism $r_\cR:\cR\to\cR^\ast$, 
$A\mapsto\langle A,\,\cdot\,\rangle$ restricts to a real linear isomorphism 
$\cH(\cR)\to\cR_\mathrm{her}^\ast$. Furthermore,
\[
\alpha:\cR_{\R,0}^\ast\to\cR_\mathrm{her}^\ast\,,
\quad
\alpha(\ell)[A+\ii B]:=\ell(A)+\ii\ell(B)\,,
\quad
A,B\in\cH(\cR)\,,
\]
is a real linear isomorphism, whose inverse is given by
$\ell\mapsto\Re\circ\ell$. So, the following diagram commutes.
(Note that two arrows in opposite directions denote a bijection.)
\[
\xymatrix{%
\cH(\cR) \ar@<.5ex>[dr]^{r_\cR} \ar@<-.5ex>[d]_{r_{\R,\cR}} \\
\cR_{\R,0}^\ast \ar@<.5ex>[r]^{\alpha} \ar@<-.5ex>[u]
 & \cR_\mathrm{her}^\ast \ar@<.5ex>[l]^{\Re} \ar@<.5ex>[ul]
}\]
Moreover, if $\ell_1\in\cR_{\R,0}^\ast$ and $\ell_2\in\cR_\mathrm{her}^\ast$
satisfy $\ell_2=\alpha(\ell_1)$, then
$\ell_1\in\cS_\R(\cR)$ holds if and only if $\ell_2\in\cS(\cR)$. Hence, 
the claim follows from Lemma~\ref{lem:riesz-real}.
\end{proof}
\begin{exa}\label{ex:Bloch}
Real *-subalgebras of $\M_n$ have a richer class of state spaces than the 
complex ones. The \emph{Bloch ball} \cite[Sec.~5.2]{BengtssonZyczkowski2017} 
\[\textstyle
\cD(\M_2)=\big\{\frac{1}{2}(\id_2+c_XX+c_YY+c_ZZ)
\colon c_X,c_Y,c_Z\in\R, c_X^2+c_Y^2+c_Z^2\leq 1\big\}
\]
is a three-dimensional Euclidean ball of radius $1/\sqrt{2}$. The set of 
density matrices of the algebra $\M_2(\R)$ of real $2\times 2$ matrices 
is the great disk 
\[\textstyle
\cD(\M_2(\R))=\big\{\frac{1}{2}(\id_2+c_XX+c_ZZ)
\colon c_X,c_Z\in\R, c_X^2+c_Z^2\leq 1\big\}
\]
of the Bloch ball $\cD(\M_2)$. There is no complex *-subalgebra of $\M_n$ 
in any dimension $n$ whose state space is a disk.
\end{exa}
\begin{rem}\label{rem:realvscomplex}
Real and complex matrix systems have the same families of state spaces. If 
$\cR$ is a real or complex matrix system on $\C^n$, then the real matrix 
system of hermitian matrices $\cH(\cR)$ has the same set of hermitian 
matrices and hence the same set of density matrices as $\cR$. Conversely, 
any real matrix system $\cR$ included in $\cH(\M_n)$ has the same same set 
of hermitian matrices and the same set of density matrices as the complex 
matrix system $\cR\oplus\ii\cR$. The set of density matrices of a real
or complex matrix system $\cR$ is affinely isomorphic to the state space of
$\cR$ by Lemma~\ref{lem:riesz-real} or~\ref{lem:riesz}, respectively.
\end{rem}
Let $\cR_1,\cR_2$ be real matrix systems on $\C^n$ and let $\cR_2\subset\cR_1$. 
We abbreviate
$\cH_i:=\cH(\cR_i)$,
$\cC_i:=\cC(\cR_i)$,
$\cC_i^\vee:=\cC^\vee(\cR_i)$,
$\cD_i:=\cD(\cR_i)$,  
$\cS_{\R,i}:=\cS_\R(\cR_i)$,
$\cS_i:=\cS(\cR_i)$, 
$r_{\R,i}:=r_{\R,\cR_i}$, and
$r_i:=r_{\cR_i}$, $i=1,2$.
Usually, the orthogonal projection of $\cR_1$ onto $\cR_2$ is the idempotent 
self-adjoint linear map $\cR_1\to\cR_1$ whose range is $\cR_2$. Reducing the 
codomain to the range, we get a map
\begin{equation}\label{eq:pi-red1}
\pi:\cR_1\to\cR_2\,,
\end{equation}
whose value at $A\in\cR_1$ is specified by the equations
$\Re\langle A-\pi(A),B\rangle=0$ for all $B\in\cR_2$. We refer to $\pi$ as 
the \emph{orthogonal projection} of $\cR_1$ onto $\cR_2$ in this paper. The 
notation $\cR_1\to\cR_2$ conveying domain and range is useful especially in 
Sec.~\ref{sec:proof-thm}. The adjoint of $\pi$ is the embedding 
$\cR_2\to\cR_1$, $A\mapsto A$. If $\cR_1$ and $\cR_2$ are complex matrix 
systems, then the Frobenius inner product induces the same orthogonal 
projection as the Euclidean scalar product \eqref{eq:scalarp}.
\par
As $\cR_i$ is the orthogonal direct sum of the spaces of its hermitian and 
skew-hermitian matrices, $i=1,2$, the map \eqref{eq:pi-red1} restricts to 
$\pi:\cH_1\to\cH_2$, which we denote (aware of the notational imprecision) 
by the same symbol $\pi$. The value of $\pi$ at $A\in\cH_1$ is specified by 
$\langle A-\pi(A),B\rangle=0$ for all $B\in\cH_2$.
\par
A \emph{cone} in a  Euclidean space 
$(E,\langle\!\langle\cdot,\cdot\rangle\!\rangle)$ is a subset $C$ of $E$ 
that is closed under multiplication with positive 
scalars. A \emph{base} of a cone $C$ is a subset $B$ of $C$ such that 
$0\not\in\aff B$ and such that for all nonzero $x\in C$ there exist 
$y\in B$ and $s>0$ such that $x=s y$ holds. Note that we have 
$B=C\cap\aff B$ for every base $B$ of a cone $C$. The set 
$M^\vee:=\{x\in E\mid\forall y\in M\colon 
\langle\!\langle x,y\rangle\!\rangle\geq 0\}$ 
is a closed convex cone for every subset $M\subset E$, called the 
\emph{dual cone} to $M$. Regarding duality of convex cones, we refer 
to \cite[Sec.~14]{Rockafellar1970}.
\par
\begin{lem}\label{lem:dualcones}
Let $\cR_1,\cR_2$ be real matrix systems on $\C^n$ such that $\cR_2\subset\cR_1$
and let $\pi:\cH_1\to\cH_2$ denote the orthogonal projection.
Then $\cC_2^\vee=\pi(\cC_1^\vee)$ and $\cD_2=\pi(\cD_1)$ holds.
\end{lem}
\begin{proof}
This lemma and its proof are similar to \cite[Prop.~5.2]{Plaumann-etal2021}.
The inclusions $\pi(\cC_1^\vee)\subset\cC_2^\vee$ and 
$\cC_2\supset(\pi(\cC_1^\vee))^\vee$ are straightforward to verify and imply
\[
\pi(\cC_1^\vee)
\subset
\cC_2^\vee
\subset 
[\pi(\cC_1^\vee)]^\vee{}^\vee\,.
\]
Since $[\pi(\cC_1^\vee)]^\vee{}^\vee$ is the closure of the convex cone 
$\pi(\cC_1^\vee)$, it suffices to show that $\pi(\cC_1^\vee)$ is closed. 
By \cite[Lemma~3.1]{Plaumann-etal2021}, this would follow if $\pi(\cD_1)$ 
was a compact base of $\pi(\cC_1^\vee)$. As $\id_n$ lies in the interior 
of $\cC_1$ in the topology of $\cH_1$, we know that $\cD_1$ is a compact 
base of $\cC_1^\vee$, see \cite[Lemma~3.3]{Plaumann-etal2021}. Hence 
$\pi(\cD_1)$ is a compact base of $\pi(\cC_1^\vee)$ provided we establish 
that $0\not\in\aff\pi(\cD_1)$. But this is clear from $\pi$ being 
trace-preserving: 
\[
\tr(\pi(A))
=\langle\id_n,\pi(A)\rangle
=\langle\id_n,A\rangle
=\tr(A)\,,
\quad
A\in\cH_1\,.
\]
This completes the proof of $\cC_2^\vee=\pi(\cC_1^\vee)$. The identity 
$\cC_2^\vee=\pi(\cC_1^\vee)$ and the fact that $\pi$ is trace-preserving 
imply $\cD_2=\pi(\cD_1)$. 
\end{proof}
Lemma~\ref{lem:dualcones} shows that the orthogonal projection 
$\pi:\cR_1\to\cR_2$ restricts to the map 
\begin{equation}\label{eq:pi-red2}
\pi:\cD_1\to\cD_2\,,
\end{equation}
which we call the \emph{orthogonal projection} of $\cD_1$ onto $\cD_2$.
To avoid any possible confusion, the orthogonal projections $\cR_1\to\cR_2$ and 
$\cH_1\to\cH_2$ are written without function labels from here on 
(Lemma~\ref{lem:dir-sum-proj} is an exception). 
\par
\begin{prop}\label{pro:restriction}
Let $\cR_1,\cR_2$ be real matrix systems on $\C^n$ such that 
\mbox{$\cR_2\subset\cR_1$}
and let $\pi:\cD(\cR_1)\to\cD(\cR_2)$ denote the orthogonal projection.
\begin{enumerate}
\item[a)]
The diagram~a) below commutes. Both $\pi:\cD_1\to\cD_2$ and 
$\cS_{\R,1}\to\cS_{\R,2}$, $\ell\mapsto\ell|_{\cR_2}$ are surjective maps. 
\item[b)]
If $\cR_1,\cR_2$ are complex matrix systems on $\C^n$, then the diagram~b) 
below commutes and the maps $\pi:\cD_1\to\cD_2$ and $\cS_1\to\cS_2$, 
$\ell\mapsto\ell|_{\cR_2}$ are onto.
\end{enumerate}
\[
\mbox{a)}\quad
\xymatrix{%
\cD_1 \ar@<-.5ex>[r]_{r_{\R,1}} \ar[d]_{\pi} 
 & \cS_{\R,1} \ar@<-.5ex>[l] \ar[d]^{\ell\mapsto\ell|_{\cR_2}}\\
\cD_2 \ar@<.5ex>[r]^{r_{\R,2}} 
 & \cS_{\R,2} \ar@<.5ex>[l]
}
\qquad\qquad
\mbox{b)}\quad
\xymatrix{%
\cD_1 \ar@<-.5ex>[r]_{r_1} \ar[d]_{\pi} 
 & \cS_1 \ar@<-.5ex>[l] \ar[d]^{\ell\mapsto\ell|_{\cR_2}}\\
\cD_2 \ar@<.5ex>[r]^{r_2} 
 & \cS_2 \ar@<.5ex>[l]
}\]
\end{prop}
\begin{proof}
a) The horizontal arrows of diagram a) and b) are obtained in 
Lemma~\ref{lem:riesz-real} and Lemma~\ref{lem:riesz}, respectively, and 
$\pi:\cD_1\to\cD_2$ is onto by Lemma~\ref{lem:dualcones}. The diagram~a) 
commutes because for all $\rho\in\cD_1$ and $A\in\cR_2$ we have
\begin{align*}
[r_{\R,2}\circ\pi(\rho)](A)
 =\Re\langle\pi(\rho),A\rangle 
 =\Re\langle\rho,A\rangle
 =[r_{\R,1}(\rho)](A)
 =r_{\R,1}(\rho)|_{\cR_2}(A)\,.
\end{align*}
Therefore and since $\pi:\cD_1\to\cD_2$ is surjective, the map
$\cS_{\R,1}\to\cS_{\R,2}$ is surjective as well. The proof of b) is similar. 
\end{proof}
\begin{cor}\label{cor:restriction}
Let $\cR_1,\cR_2$ be real matrix systems on $\C^n$ such that $\cR_2\subset\cR_1$,
let $\pi:\cD(\cR_1)\to\cD(\cR_2)$ denote the orthogonal projection, and let 
$\rho\in\cD_1$. 
\begin{enumerate}
\item[a)] 
The map $\pi:\cD_1\to\cD_2$ is open at $\rho$ if and only if 
$\cS_{\R,1}\to\cS_{\R,2}$, $\ell\mapsto\ell|_{\cR_2}$ is open at 
$r_{\R,1}(\rho)$.
\item[b)]
If $\cR_1,\cR_2$ are complex matrix systems on $\C^n$, then $\pi:\cD_1\to\cD_2$ 
is open at $\rho$ if and only if $\cS_1\to\cS_2$, $\ell\mapsto\ell|_{\cR_2}$ is 
open at $r_1(\rho)$.
\end{enumerate}
\end{cor}
\begin{proof}
This follows directly from Prop.~\ref{pro:restriction}.
\end{proof}
\begin{rem}\label{rem:intersection}
Let $\cR_1,\cR_2$ be real matrix systems on $\C^n$ and let $\cR_2\subset\cR_1$. 
Then $\cD_2\supset\cD_1\cap\cR_2$ holds, but the converse inclusion 
is wrong in general, as Ex.~\ref{ex:ex1a} shows. However, if $\cR_2$ is a real 
*-subalgebra of $\M_n$, then we have 
\begin{equation}\label{eq:D-inter}
\cD_2=\cD_1\cap\cR_2\,.
\end{equation}
Indeed, $\cC_2^\vee=\cC_2$ holds \cite[Thm.~III.2.1]{FarautKoranyi1994} as the 
space of hermitian matrices $\cH_2$ is a Euclidean Jordan algebra with Jordan 
product $A\circ B=\frac{1}{2}(AB+BA)$ and inner product 
$(A,B)\mapsto\Re\langle A,B\rangle$, $A,B\in\cH_2$. So,
$\cC_2^\vee=\cC_2\subset\cC_1\subset\cC_1^\vee$ proves $\cD_2\subset\cD_1$, 
which implies \eqref{eq:D-inter}. See also the Notes to Chapter~6 in 
\cite{AlfsenShultz2003}.
\end{rem}
\begin{exa}\label{ex:ex1a}
Despite $\cR_2\subset\cR_1\subset\M_3$, the inclusions 
$\cD_2\subset\cD_1\subset\cD(\M_3)$ fail if $\cR_2:=\spn_\C(\id_3,Z\oplus 0)$ 
and $\cR_1:=\spn_\C(\id_3,X\oplus 1,Z\oplus 0)$. We have
\[\textstyle
A_\lambda\in\cD_2\Leftrightarrow|\lambda|\leq\frac{3}{2}\,,
\quad
A_\lambda\in\cD_1\Leftrightarrow|\lambda|\leq\sqrt{2}\,,
\quad
A_\lambda\in\cD(\M_3)\Leftrightarrow|\lambda|\leq 1
\]
for $A_\lambda:=(\id_3+\lambda\,Z\oplus 0)/3\in\cR_2$, $\lambda\in\R$.
The second equivalence is obtained by minimizing $\langle A_\lambda,A\rangle$ 
for fixed $\lambda$ over $A\in\cC_1$, that is, by minimizing 
$\langle A_\lambda,\id_3+c_1(X\oplus 1)+c_2Z\oplus 0\rangle
=1+\frac{1}{3}c_1+\frac{2}{3}\lambda c_2$ 
on the unit disk of $\R^2$.
\end{exa}
%
%
\section{Direct convex sums}
\label{sec:directsums}
This section addresses affinely independent convex sets, their convex hulls,
and maps defined thereon. Let
\[
\Delta_m:=\{(s_1,\dots,s_m)\in\R^m\mid\forall i: s_i\geq0, s_1+\dots+s_m=1\}
\]
denote the probability simplex, and 
\[\textstyle
\Delta_m(\epsilon,s_1,\dots,s_m)
:=\Delta_m\cap\bigoplus_{i=1}^m(s_i-\epsilon,s_i+\epsilon)
\]
the open hypercube of edge length $2\epsilon$ centered at 
$(s_1,\dots,s_m)\in\Delta_m$.
\par
A family of convex subsets $K_1,\dots,K_m$ of a Euclidean space is 
\emph{affinely independent} if every point in their convex hull can be 
expressed by a unique convex combination $s_1x_1+\dots+s_mx_m$. This 
means that $(s_1,\dots,s_m)\in\Delta_m$ is unique and $x_i\in K_i$ is 
unique for all $i$ for which if $s_i>0$, $i=1,\dots,m$. The 
\emph{direct convex sum} \cite{Alfsen1971} of a family $K_1,\dots,K_m$ 
of affinely independent convex sets is defined as their convex hull
\[
K_1\dcs\cdots\dcs K_m:=\conv(K_1\cup\dots\cup K_m)\,.
\]
If $K_1,\dots,K_m$ is a family of affinely independent compact convex sets,
then their direct convex sum is compact \cite[Thm.~17.2]{Rockafellar1970}. 
A compactness argument allows us to describe a base of open neighborhoods.
\par
\begin{lem}\label{lem:base-dcs}
Let $K_1,\dots,K_m$ be affinely independent compact convex subsets of a 
Euclidean space, and let $x_i\in K_i$, $i=1,\dots,m$, and  
$(s_1,\dots,s_m)\in\Delta_m$. Let $I:=\{i\in\{1,\ldots,m\}\mid s_i>0\}$,
$\delta:=\min_{i\in I}s_i$, and define
\[\begin{array}{r@{\hskip2pt}l}
O_I(\epsilon,(A_i)_{i\in I})
:=\{t_1y_1+\dots+t_my_m\mid &
(t_1,\dots,t_m)\in\Delta_m(\epsilon,s_1,\dots,s_m),\\
  & \mbox{$\forall i\colon y_i\in K_i$ and ($i\in I \Rightarrow y_i\in A_i$)} \}
\end{array}\]
for every $\epsilon\in(0,\delta]$ and open set $A_i$ in the relative 
topology of $K_i$ containing $x_i$, $i=1,\dots,m$. Then the family 
$\{O_I(\epsilon,(A_i)_{i\in I})\}$ is a local base of open neighborhoods at 
$s_1x_1+\dots+s_mx_m$ in the relative topology of $K_1\dcs\cdots\dcs K_m$.
\end{lem}
\begin{proof}
As $K:=K_1\dcs\cdots\dcs K_m$ is a metric space, it suffices to show that 
$O:=O_I(\epsilon,(A_i)_{i\in I})$ is open and there are arbitrary small 
sets of this form.
\par
Let $B_i:=A_i$ if $i\in I$ and $B_i:=K_i$ if $i\not\in I$, $i=1,\dots,m$. 
Then the set 
\[\textstyle
U:=\Delta_m(\epsilon,s_1,\dots,s_m)
\oplus\bigoplus_iB_i
\]
is open in the relative topology of the compact set 
$\widetilde{K}:=\Delta_m\oplus\bigoplus_iK_i$. The complement 
\[
U^\complement:=\widetilde{K}\setminus U
=\underbrace{\textstyle
\Delta_m(\epsilon,s_1,\dots,s_m)^\complement\oplus\bigoplus_iK_i}_{C:=}
\cup\bigcup_{j\in I}
\underbrace{\textstyle
\Delta_m\oplus(\bigoplus_{i\neq j}K_i)\oplus A_j^\complement}_{C_j:=}
\]
is compact. Hence, its image under the continuous surjective map
\[\textstyle
\beta:\widetilde{K}\to K\,,
\quad
((t_i)_{i=1}^m,(y_i)_{i=1}^m)\mapsto\sum_{i=1}^m t_iy_i
\]
is compact. We prove that $O=\beta(U)$ is open in $K$ by showing that 
$\beta(U)$ is disjoint from $\beta(U^\complement)$. Let
\[
u:=(t_1,\dots,t_m,y_1,\dots,y_m)\in U\,.
\]
Then $\beta(u)\not\in\beta(C)$ follows as the vector $(t_i)_{i=1}^m$ is 
uniquely determined by $\beta(u)$. For every $j\in I$ we have $s_j>0$, 
hence $t_j>0$ by the definition of $O$. Thus, $y_j$ is uniquely determined 
by $\beta(u)$, which proves $\beta(u)\not\in\beta(C_j)$.
\par
Let $|M|:=\sup_{x,y\in M}|y-x|$ denote the diameter of a set $M$. The 
distance of $x:=s_1x_1+\dots+s_mx_m$ from a point $t_1y_1+\dots+t_my_m$ 
in $O$ is bounded by 
\[\textstyle
\sum_i \left( |s_i-t_i| |x_i| + t_i |x_i-y_i| \right)
\leq \epsilon\sum_i
|x_i| + \sum_{i\in I}t_i|A_i| + \epsilon\sum_{i\not\in I}|K_i|\,.
\]
Choosing open sets $(A_i)_{i\in I}$ with diameters at most $\epsilon$, 
we obtain 
\[\textstyle
|O|\leq 2\sup_{y\in O}|x-y|\leq 2\epsilon\left(
\sum_i|x_i| + 1 + \sum_{i\not\in I}|K_i|
\right)\,,
\]
which completes the proof, as the compact sets $K_i$ have finite 
diameters and as $\epsilon$ can be chosen arbitrarily small.
\end{proof}
Let $E_1,\dots,E_m$ be Euclidean spaces. We consider $E_i$ as a 
subspace of the direct sum $\bigoplus_{j=1}^m E_j$ \emph{via} the 
embedding
\[\textstyle
E_i\to\bigoplus_{j=1}^m E_j\,,
\quad 
x\mapsto (\underbrace{0,\dots,0}_{\makebox[0pt]{\scriptsize $i-1$ zeros}},x,
\underbrace{0,\dots,0}_{\makebox[0pt]{\scriptsize $m-i$ zeros}})\,,
\quad
i=1,\dots,m\,. 
\] 
\par
\begin{prop}\label{pro:dcs-map}
Let $K_i\subset E_i$ and $L_i\subset F_i$ be compact convex subsets of Euclidean 
spaces $E_i$ and $F_i$, $i=1,\dots,m$, such that $K_1,\dots,K_m$ are affinely 
independent in $\bigoplus_{i=1}^mE_i$ and $L_1,\dots,L_m$ are affinely 
independent in $\bigoplus_{i=1}^mF_i$. If $f_i:K_i\to L_i$ is a map, 
$i=1,\dots,m$, then a map 
\begin{equation}\label{eq:dcs-map}\textstyle
f_1\dcs\cdots\dcs f_m:K_1\dcs\cdots\dcs K_m \to L_1\dcs\cdots\dcs L_m
\end{equation}
is well defined by
\[\textstyle
s_1x_1+\dots+s_mx_m \mapsto s_1f_1(x_1)+\dots+s_mf_m(x_m)\,.
\]
If $f_i$ is open and surjective for $i=1,\dots,m$, then $f_1\dcs\cdots\dcs f_m$ 
is open.
\end{prop}
\begin{proof}
The map is well defined as the sets $K_1,\dots,K_m$ are affinely independent
and the sets $L_1,\dots,L_m$ are affinely independent. Regarding the openness
of $f_1\dcs\cdots\dcs f_m$, it suffices to show that the images of the members 
of a base of the relative topology of $K_1\dcs\cdots\dcs K_m$ are open. 
Using the base of Lemma~\ref{lem:base-dcs}, we have
\[
f_1\dcs\cdots\dcs f_m[O_I(\epsilon,(A_i)_{i\in I})]
=O_I[\epsilon,(f(A_i))_{i\in I}]\,,
\]
which is open in $L_1\dcs\cdots\dcs L_m$. The required identity of $f_i(K_i)=L_i$ 
for every $i\not\in I$ is a consequence of $f_i$ being surjective.
\end{proof}
We call the map \eqref{eq:dcs-map} the \emph{direct convex sum} of the maps 
$f_1,\dots,f_m$. Next, we recall a sufficient condition for the affine 
independence of convex sets \cite{ClausingPapadopoulou1978}.
\par
\begin{rem}\label{rem:embedded-aff-ind}
Let $K_i$ be a convex subset of a Euclidean space $E_i$ such that 
$0\not\in\aff K_i$ holds for $i=1,\dots,m$. Then $K_1,\dots,K_m$ 
are affinely independent in the direct sum $\bigoplus_{i=1}^m E_i$.
\end{rem}
Returning to matrix systems, we consider a real matrix system $\cR_i$ 
on $\C^{n_i}$, $i=1,\dots,m$. The direct sum $\cR:=\bigoplus_{i=1}^m\cR_i$ is 
a real matrix system on $\C^{n_1+\dots+n_m}$. The Frobenius inner product of 
$(A_i)_{i=1}^m,(B_i)_{i=1}^m\in\cR$ is 
$\langle (A_i)_{i=1}^m,(B_i)_{i=1}^m\rangle=\sum_{i=1}^m\langle A_i,B_i\rangle$.
\par
\begin{lem}\label{lem:dir-sum}
Let $\cR_i$ be a real matrix system on $\C^{n_i}$, $i=1,\ldots,m$. Then 
\begin{align*}
\cC(\cR_1\oplus\dots\oplus\cR_m) &=\cC(\cR_1)\oplus\dots\oplus\cC(\cR_m)\,,\\
\cC^\vee(\cR_1\oplus\dots\oplus\cR_m) 
 &=\cC^\vee(\cR_1)\oplus\dots\oplus\cC^\vee(\cR_m)\,,\\
\text{and}\quad
\cD(\cR_1\oplus\dots\oplus\cR_m) &=\cD(\cR_1)\dcs\cdots\dcs\cD(\cR_m)\,.
\end{align*}
\end{lem}
\begin{proof}
The first identity is clear. The second one follows by induction from $m=2$,
a case that is easy to verify. By Rem.~\ref{rem:embedded-aff-ind} and because 
$\cD(\cR_i)$ is included in the hyperplane of trace-one matrices $i=1,\dots,m$, 
the sets $\cD(\cR_1),\dots,\cD(\cR_m)$ are affinely independent in the direct 
sum $\bigoplus_{i=1}^m\cH(\cR_i)$. The third identity follows from the second 
one by enforcing the trace to be one. 
\end{proof}
\begin{exa}\label{ex:ex1b1}
Let $\cR:=\spn_\R(\id_3,X\oplus 1,Z\oplus 0)$.
\begin{enumerate}
\item[a)] 
The set of density matrices $\cD(\M_2\oplus\M_1)$ of the *-sub\-al\-ge\-bra 
$\M_2\oplus\M_1$ of $\M_3$ is a symmetric cone that is the direct convex 
sum of the Bloch ball $\cD(\M_2)$ and the singleton $\cD(\M_1)=\{1\}$, see 
Lemma~\ref{lem:dir-sum} and Ex.~\ref{ex:Bloch}. The closed segment 
\[\textstyle
\cG:=
\big[\ket{+}\!\!\bra{+}\oplus0,0\oplus 1\big]
=\big\{(1-\lambda)\ket{+}\!\!\bra{+}\oplus\lambda
 \colon \lambda\in[0,1] \big\}
\] 
is a generatrix of this cone. 
One proves along the lines of Ex.~\ref{ex:ex1} that the orthogonal projection 
\[\textstyle
\pi:\cD(\M_2\oplus\M_1)\to\cD(\cR)
\]
is not open at any point in the half-open segment 
\[\textstyle
\cG_0
:=\big(\ket{+}\!\!\bra{+}\oplus0,0\oplus 1\big]
=\cG\setminus\{\ket{+}\!\!\bra{+}\oplus 0\}\,.
\] 
Thereby, the equivalence of states and density matrices is described in 
Coro.~\ref{cor:restriction}. It is important to observe that 
$\frac{1}{2}(\id_2+\cos(\theta)X+\sin(\theta)Z)\oplus 0$ and 
$\cos(\theta)(X\oplus 1)+\sin(\theta)Z\oplus 0$ are matrices in $\M_2\oplus\M_1$ 
for all $\theta\in\R$. The map $\pi$ is open at every point in the complement 
of $\cG_0$ because $\cD(\M_2\oplus\M_1)$ is a cone over a ball, see
\cite[Lemma~4.17]{Weis2014} for a detailed proof.
\item[b)]
The orthogonal projection $\pi:\cD(\M_2\oplus\M_1)\to\cD(\cR)$ has an
instructive geometry. The set of density matrices
$\cD(\cR)=\pi\big(\cD(\M_2)\dcs\{1\}\big)$ is the convex hull of the 
projected ball $\pi(\cD(\M_2)\oplus\{0\})$ and the singleton 
$\pi(0\oplus 1)$ by Lemma~\ref{lem:dualcones}. In turn, 
$\pi(\cD(\M_2)\oplus\{0\})$ is the filled ellipse of all points
\begin{align}
\label{eq:piR}
&\textstyle 
\hphantom{{}={}}\pi\big(\frac{1}{2}(\id_2+c_XX+c_YY+c_ZZ)\oplus 0\big)\\
\nonumber
&\textstyle 
=\frac{1}{2}\big(%
\id_3-M
+c_X(3M-\id_3)
+c_ZZ\oplus0\big)\,,
\end{align}
where $c_X,c_Y,c_Z\in\R$ satisfy $c_X^2+c_Y^2+c_Z^2=1$, and 
$M:=\frac{1}{2}(\ket{+}\!\!\bra{+}\oplus 1)$ is the midpoint of the generatrix 
$\cG$. Note that $\id_3$, $(3X-\id_2)\oplus2$, $Z\oplus 0$ is an orthogonal 
basis of $\cR$. The choice of $c_X=1,c_Y=c_Z=0$ yields
\begin{equation}\label{eq:fiberM2M1}\textstyle
\pi(\ket{+}\!\!\bra{+}\oplus0)=M=\pi(0\oplus 1)\,.
\end{equation}
Thus, $\cG$ is perpendicular to $\cR$, and  $\cD(\cR)=\pi(\cD(\M_2)\oplus\{0\})$ 
is an ellipse. Moreover, $\pi^{-1}(M)=\cG$ holds by equation~\eqref{eq:piR} 
and~\eqref{eq:fiberM2M1}. 
\item[c)]
The ellipse $\cD(\cR)$ has the semiaxes $\sqrt{3/8}$ and $1/\sqrt{2}$. This
follows from the formula \eqref{eq:piR} when $(c_X,c_Y,c_Z)$ is assigned the
values of $(\pm1,0,0)$ and $(0,0,\pm1)$. In retrospect to Rem.~\ref{ex:ex1a}, 
the value $(1/3,0,\sqrt{8/9})$ confirms that $(\id_3+\lambda\,Z\oplus 0)/3$ 
is contained in $\cD(\cR)$ if and only if $|\lambda|\leq\sqrt{2}$. 
\end{enumerate}
\end{exa}
It is instructive to look at Ex.~\ref{ex:ex1b1} from the perspective of a
real *-subalgebra of $\M_3$, whose state space can be visualized in 
three-space.
\par
\begin{exa}\label{ex:ex1b2}
The set of density matrices $\cD(\M_2(\R)\oplus\M_1(\R))$ of the real 
*-subalgebra $\M_2(\R)\oplus\M_1(\R)$ of $\M_3$ is a symmetric cone, which 
is the direct convex sum of the great disk $\cD(\M_2(\R))$ of the Bloch ball 
and a singleton. As in Ex.~\ref{ex:ex1b1}, the closed segment $\cG$ is a 
generatrix of this cone, which is the fiber of the orthogonal 
projection 
\[\textstyle
\pi:\cD(\M_2(\R)\oplus\M_1(\R))\to\cD(\cR)
\]
over $M$. The map $\pi$ is not open at any point in 
$\cG_0=\cG\setminus\{\ket{+}\!\!\bra{+}\oplus 0\}$ and open at every point 
in the complement of $\cG_0$. The lack of openness can be visualized 
graphically in three-space by observing that $\pi$ projects the cone 
$\cD(\M_2(\R)\oplus\M_1(\R))$ along its generatrix $\cG$ to the ellipse 
$\cD(\cR)$. 
\end{exa}
\begin{lem}\label{lem:dir-sum-proj}
Let $\cR_i,\cR_i'$ be real matrix systems on $\C^{n_i}$ such that
$\cR_i\subset\cR_i'$, let $\pi_i:\cR_i'\to\cR_i$ denote the orthogonal 
projection, $i=1,\ldots,m$, and consider the direct sums 
$\cR':=\bigoplus_{i=1}^m\cR_i'$ and $\cR:=\bigoplus_{i=1}^m\cR_i$. The 
orthogonal projection $\pi:\cD(\cR')\to\cD(\cR)$ equals
\[
\pi_1\dcs\cdots\dcs\pi_m:
\cD(\cR_1')\dcs\cdots\dcs\cD(\cR_m')
\to\cD(\cR_1)\dcs\cdots\dcs\cD(\cR_m)\,.
\]
\end{lem}
\begin{proof}
The orthogonal projections $\pi_i:\cD(\cR_i')\to\cD(\cR_i)$, $i=1,\ldots,m$, 
and $\pi:\cD(\cR')\to\cD(\cR)$ are well defined by Lemma~\ref{lem:dualcones}. 
A straight-forward computation shows that the orthogonal projection 
$\pi:\cR'\to\cR$ is the direct sum $\bigoplus_{i=1}^m\pi_i$. The claim then 
follows from the third identity of Lemma~\ref{lem:dir-sum} and from the 
definition of the direct convex sum of maps in formula 
\eqref{eq:dcs-map}.
\end{proof}
%
%
\section{Numerical ranges}
\label{sec:numerical-ranges}
The orthogonal projection $\cD(\M_n)\to\cD(\cR)$ may be restricted to the set 
$\ex\cD(\M_n)$ of extreme points\footnote{%
An point in a convex set $K$ is an \emph{extreme point} \cite{Rockafellar1970} 
of $K$ if there is no way to express it as a convex combination  
$(1-s)x+sy$ such that $x,y\in K$ and $s\in(0,1)$, except by taking $x=y$.
We denote the set of extreme points of $K$ by $\ex K$.} 
of $\cD(\M_n)$. The openness of this restriction was studied in matrix theory 
\cite{Corey-etal2013,Leake-etal2014-preimages,Leake-etal2014-inverse,
Leake-etal2016-inverse,LinsParihar2016,LinsSpitkovsky2020,
SpitkovskyWeis2016,SpitkovskyWeis2018,Weis2016} for two-dimensional 
state spaces $\cD(\cR)$ represented as numerical ranges
(see Rem.~\ref{rem:eigenfunctions}). 
\par
Let $A_1,\ldots,A_k\in\cH(\M_n)$ and consider the real matrix system 
\[\textstyle
\cR(A_1,\ldots,A_k):=\spn_\R(\id_n,A_1,\ldots,A_k)\,.
\]
The image of $\cD(\M_n)$ under the real linear map 
\[\textstyle
v:\cH(\M_n)\to\R^k\,, 
\quad
B\mapsto(\langle B,A_1\rangle,\dots,\langle B,A_k\rangle)\tp
\]
is the \emph{joint numerical range} $V(A_1,\ldots,A_k):=v(\cD(\M_n))\subset\R^k$,
see \cite{BonsallDuncan1971}. Aware of the notational imprecision, we use the 
same label $v$ also for several restrictions of $v$, among others for
\begin{align}
\label{eq:proj-jnr2}
v:\cD(\M_n)&\to V(A_1,\ldots,A_k)\,, 
& \rho & \mapsto(\langle\rho,A_1\rangle,\dots,\langle\rho,A_k\rangle)\tp\,.
\end{align}
\par
\begin{lem}\label{lem:jnr}
The following diagram commutes.
\[\textstyle
\xymatrix{%
 & \cD(\M_n) \ar[d]^\pi \ar@<-1.0ex>[dl]_v \\
*+[l]{V(A_1,\ldots,A_k)} \ar@<-.5ex>[r]
 & *+[r]{\cD(\cR(A_1,\ldots,A_k))} \ar@<-.5ex>[l]_v
}\]
\end{lem}
\begin{proof}
Let $\cR:=\cR(A_1,\ldots,A_k)$. By Lemma~\ref{lem:dualcones}, the orthogonal 
projection $\pi:\cD(\M_n)\to\cD(\cR)$ is surjective. It is straightforward to 
verify that $v:\cH(\M_n)\to\R^k$ factors through $\cR$, in the sense 
that $v=v\circ\pi$ holds. Hence, the commutativity of the diagram is 
implied by the injectivity of $v$ restricted to the affine space 
$\{B\in\cH(\cR):\tr(B)=1\}$. Let $B_1,B_2$ be contained in this affine 
space. If $v(B_1)=v(B_2)$, then 
\begin{align*}
0 &= v(B_1-B_2)=\langle B_1-B_2,A_i\rangle_{i=1}^k\\
\text{and}\quad
0 &= \tr(B_1)-\tr(B_2)=\langle B_1-B_2,\id_n\rangle\,.
\end{align*}
This implies $B_1-B_2=0$ as 
$B_1-B_2\in\cH(\cR)=\spn_\R(\id_n,A_1,\ldots,A_k)$.
\end{proof}
In the remainder of this section, let $k=2$ and $A:=A_1+\ii A_2$. The image 
of the unit sphere $\C S^n:=\{\ket{\varphi}\in\C^n:\braket{\varphi|\varphi}=1\}$ 
under the hermitian quadratic form 
$f_A:\C^n\to\C$, $\ket{\varphi}\mapsto\braket{\varphi|A\varphi}$ is the 
\emph{numerical range} 
\[\textstyle
W(A):=f_A(\C S^n)\subset\C\,.
\]
Here, 
$\braket{\varphi_1|\varphi_2}:=\overline{x_1}y_1+\cdots+\overline{x_n}y_n$ 
is the inner product of $\ket{\varphi_1}=(x_1,\ldots,x_n)\tp$ and 
$\ket{\varphi_2}=(y_1,\ldots,y_n)\tp$ in $\C^n$. We use the same label 
$f_A$ to denote the map
\[
f_A:\C S^n\to W(A)\,,
\quad 
\ket{\varphi}\mapsto\braket{\varphi|A\varphi}\,.
\]
\par
Minkowski's theorem \cite{Schneider2014} asserts that every compact convex set 
is the convex hull of its extreme points.
\par
\begin{prop}\label{pro:nr}
The following diagram commutes. 
\[\textstyle
\xymatrix{%
\C S^n \ar[rr]^-{\ket{\varphi}\mapsto\ket{\varphi}\!\bra{\varphi}} \ar[d]_{f_A} 
 && \ex\cD(\M_n) \ar@<0.0ex>[d]_v \ar@<0.0ex>[r]^-{\rho\mapsto\rho}
 & *+[r]{\cD(\M_n)} \ar@<0.0ex>[dl]_v \ar[d]^\pi\\
W(A) \ar@<.5ex>[rr] 
 && V(A_1,A_2) \ar@<-.5ex>[r] \ar@<.5ex>[ll]_-%
{\raisebox{10pt}{\scriptsize$z\mapsto
\left(\begin{smallmatrix}\Re(z)\\\Im(z)\end{smallmatrix}\right)$}}
 & *+[r]{\cD(\cR(A_1,A_2))} \ar@<-.5ex>[l]_-v
}\]
\end{prop}
\begin{proof}
The bottom right triangle is the case $k=2$ of Lemma~\ref{lem:jnr}. The 
map $f_A:\C S^n\to W(A)$ factors through $\ex\cD(\M_n)$, $\cD(\M_n)$, 
and $V(A_1,A_2)$, as $\C S^n\to\cD(\M_n)$, 
$\ket{\varphi}\mapsto\ket{\varphi}\!\!\bra{\varphi}$ parametrizes the 
extreme points of $\cD(\M_n)$, see for example 
\cite[Sec.~5.1]{BengtssonZyczkowski2017}, and since for all 
$\ket{\varphi}\in\C S^n$ we have
\begin{align*}
f_A(\ket{\varphi})
&=\braket{\varphi|A\varphi}
=\tr(\ket{\varphi}\!\!\bra{\varphi}A)
=\left\langle\ket{\varphi}\!\!\bra{\varphi},A\right\rangle\\
&=\left\langle\ket{\varphi}\!\!\bra{\varphi},A_1\right\rangle
+\ii\left\langle\ket{\varphi}\!\!\bra{\varphi},A_2\right\rangle\,.
\end{align*}
It remains to show that $g:W(A)\to V(A_1,A_2)$, $z\mapsto(\Re(z),\Im(z))\tp$ 
is onto (being the restriction of a bijection, the map $g$ is one-to-one). 
\par
First, we show that $\ex V(A_1,A_2)$ is included in the image of $g$. The 
preimage of every extreme point $x$ of $V(A_1,A_2)$ under 
$v:\cD(\M_n)\to V(A_1,A_2)$ contains an extreme point of $\cD(\M_n)$. This 
is true because the preimage of $x$ is a face $\cF$ of $\cD(\M_n)$, which 
has an extreme point $\rho$ by Minkowski's theorem, since $\cD(\M_n)$, and 
hence $\cF$, is compact. Since $\rho$ is also an extreme point of $\cD(\M_n)$, 
the claim follows from $\C S^n\to\ex\cD(\M_n)$ being onto. 
\par
Second, the convex hull of $\ex V(A_1,A_2)$ is included in the image of $g$
because $W(A)$ is convex by the Toeplitz-Hausdorff theorem 
\cite{Davis1971,MaierNetzer2024}. Third, the map $g$ is onto, because 
$V(A_1,A_2)$ is the convex hull of its extreme points, again by Minkowski's 
theorem.
\end{proof}
The affine isomorphism $W(A)\cong\cD(\cR(A_1,A_2))$ of Prop.~\ref{pro:nr} 
has an analogue in the much more general setting of matrix-valued states 
\cite[Thm.~5.1]{Farenick2004}. 
\par
Let $f_A^{-1}$ denote the multi-valued inverse of $f_A:\C S^n\to W(A)$. Corey 
et al. \cite{Corey-etal2013} define $f_A^{-1}$ to be \emph{strongly continuous} 
at $z\in W(A)$ if the map $f_A$ is open at every point in the fiber 
$f_A^{-1}(z)$ of $f_A$ over $z$. 
\par
\begin{rem}\label{rem:eigenfunctions}
Strong continuity can be described in terms of standard results on the 
numerical range. We refer to Sec.~8 of \cite{Leake-etal2014-preimages} 
and the references therein. 
\par
There exists a family of orthonormal bases 
$\ket{\varphi_1(\theta)},\dots,\ket{\varphi_n(\theta)}$ of $\C^n$ that is analytically 
parametrized by a real number $\theta\in\R$; and there are analytic functions 
$\lambda_i:\R\to\R$, called \emph{eigenfunctions} \cite{Leake-etal2014-inverse}, 
such that
\[\textstyle
(\cos(\theta)A_1+\sin(\theta)A_2)\ket{\varphi_i(\theta)}
=\lambda_i(\theta)\ket{\varphi_i(\theta)}\,,
\quad 
\theta\in\R\,,
\quad
i=1,\dots,n\,.
\]
For every $i=1,\ldots,n$, the numerical range $W(A)$ includes the image 
$\Img(z_i)$ of the curve 
\[\textstyle
z_i:\R\to\C\,,
\quad 
\theta\mapsto e^{\ii\theta}(\lambda_i(\theta)+\ii\lambda'_i(\theta))\,.
\]
Every extreme point of $W(A)$ is contained in $\Img(z_i)$ for some 
$i=1,\ldots,n$.
\par
If $z\in W(A)$ is not an extreme point of 
$W(A)$, then $f_A^{-1}$ is strongly continuous at $z$
\cite[Thm.~4]{Corey-etal2013}. Let $z$ be an extreme point of $W(A)$. Then 
there are $\theta_0\in\R$ and $i_0\in\{1,\ldots,n\}$ such that 
$z=z_{i_0}(\theta_0)$. Now, $f_A^{-1}$ is strongly continuous at $z$ if and 
only if $z_i(\theta_0)=z$ implies $z_i=z_{i_0}$ for all $i=1,\ldots,n$, see 
\cite[Thm.~2.1.1]{Leake-etal2014-inverse}. 
Leake at al.~\cite{Leake-etal2014-inverse} state the latter condition by 
saying that the eigenfunctions corresponding to $z$ at $\theta_0$ do not split. 
\end{rem}
\begin{rem}\label{rem:pi-eigenfunctions}
Strong continuity is connected to the openness of a linear map. For all 
$z\in W(A)$, the multi-valued map $f_A^{-1}$ is strongly continuous at $z$ if 
and only if the restricted linear map $v:\cD(\M_n)\to V(A_1,A_2)$, introduced 
in \eqref{eq:proj-jnr2}, is open at every point in the fiber 
$v^{-1}[(\Re(z),\Im(z))\tp]$, see \cite[Coro.~5.2]{Weis2016}. Moreover, for 
every $x\in V(A_1,A_2)$, the map $v$ is open at some point in the relative 
interior\footnote{%
The \emph{relative interior} \cite{Rockafellar1970} of a convex set $K$, denoted 
by $\ri(K)$, is the interior of $K$ in the relative topology of the affine hull 
of $K$.} 
of $v^{-1}(x)$ if and only if $v$ is open at every point in $v^{-1}(x)$.
\end{rem}
\begin{prop}\label{pro:DR-eigenfunctions}
Let $z\in W(A)$, let $x:=(\Re(z),\Im(z))\tp\in\R^2$, let $\rho$ be 
the unique density matrix of $\cR(A_1,A_2)$ that satisfies $v(\rho)=x$,
and let $\pi:\cD(\M_n)\to\cD(\cR(A_1,A_2))$ denote the orthogonal 
projection. Then the following assertions are equivalent.
\begin{itemize}
\item
$f_A^{-1}$ is strongly continuous at $z$,
\item
$v:\cD(\M_n)\to V(A_1,A_2)$ is open at every point in $v^{-1}(x)$, 
\item
$\pi:\cD(\M_n)\to\cD(\cR(A_1,A_2))$ is open at every point in $\pi^{-1}(\rho)$.
\end{itemize}
We have 
$z\in\ex W(A)\Leftrightarrow x\in\ex V(A_1,A_2)
\Leftrightarrow \rho\in\ex\cD(\cR(A_1,A_2))$. If $z\in\ex W(A)$, then
$z=z_{i_0}(\theta_0)$ for some $\theta_0\in\R$ and $i_0\in\{1,\ldots,n\}$,
and the following assertions are equivalent.
\begin{itemize}
\item
$f_A^{-1}$ is strongly continuous at $z$,
\item
the eigenfunctions corresponding to $z$ at $\theta_0$ do not split,
\item
$\pi$ is open at some point in the relative interior of $\pi^{-1}(\rho)$.
\end{itemize}
\end{prop}
\begin{proof}
The claims follow directly from Rem.~\ref{rem:eigenfunctions} 
and~\ref{rem:pi-eigenfunctions}, and Prop.~\ref{pro:nr}.
\end{proof}
\begin{exa}\label{ex:riM2M1}
Prop.~\ref{pro:DR-eigenfunctions} ignores that $\pi:\cD(\M_n)\to\cD(\cR(A_1,A_2))$ 
could be open at some point in the fiber $\pi^{-1}(\rho)$ over an extreme point 
$\rho$ of $\cD(\cR(A_1,A_2))$, but not open anywhere in the relative interior of 
$\pi^{-1}(\rho)$.
\par
This occurs for $A_1:=X\oplus 1$ and $A_2:=Z\oplus 0$. As discussed in 
Ex.~\ref{ex:ex1b1}~a) and~b), the orthogonal projection 
$\pi_\cR:\cD(\M_2\oplus\M_1)\to\cD(\cR(A_1,A_2))$ is open at 
$\ket{+}\!\!\bra{+}\oplus0$ and nowhere else in the fiber 
$\pi_\cR^{-1}(M)=\big[\ket{+}\!\!\bra{+}\oplus0,0\oplus 1\big]$ over 
$M=\frac{1}{2}(\ket{+}\!\!\bra{+}\oplus 1)$. In particular, $\pi_\cR$ is not 
open anywhere in the relative interior of $\pi_\cR^{-1}(M)$. Ex.~\ref{ex:ex1c} 
proves the analogue for the larger fiber $\pi^{-1}(M)$.
\end{exa}
%
%
\section{Retractions}
\label{sec:retraction}
Generalizing a known result, we prove that the set of density matrices of 
every real *-subalgebra of $\M_n$ is a stable convex set. The proof relies 
on retractions of state spaces, a topic that also proved helpful in the 
foundations of quantum information theory \cite{Harremoes2018}. 
\par
\begin{rem}[Stability of state spaces]\label{rem:Dstable}~
\begin{enumerate}
\item[a)]
The stability problem of a finite-dimensional compact convex set $K$ is 
completely solved \cite{Papadopoulou1977}. The \emph{$d$-skeleton} of $K$
is defined as the union of all faces\footnote{%
A \emph{face} \cite{Rockafellar1970} of a convex set $K$ is a convex subset 
$F\subset K$ such that $x,y\in F$ is implied by $(1-s)x + sy\in F$ for all
$x,y\in K$ and $s\in(0,1)$.}  
of $K$ of dimension at most $d$. The convex set $K$ is stable if and only 
if for every nonnegative integer $d$, the $d$-skeleton of $K$ is closed.
\item[b)]
The set of density matrices $\cD(\M_n)$ is stable \cite{Debs1979} because all 
its $d$-skeletons are closed. The closedness follows from three arguments: First,
the set $\cD(\M_n)$ is a compact convex set of dimension $n^2-1$. Second, every 
nonempty face of $\M_n$ is unitarily similar to $\cD(\M_l)\oplus\{0\}$ for some 
positive integer $l$, and third, the unitary group $U(n)$ is compact. 
\item[c)]
The state space $\cS(\cA)$ of every *-subalgebra $\cA$ of $\M_n$ is stable 
\cite{Papadopoulou1982} as it is a direct convex sum of state spaces of full
matrix algebras $\M_n$. As $\cS(\cA)$ is stable, the set of density matrices 
$\cD(\cA)$ is stable, too, by Lemma~\ref{lem:riesz}.
\end{enumerate}
\end{rem}
A \emph{retraction} is an affine map $f:K\to L$ between compact convex sets 
$K,L$ which is left-inverse to an affine map $g:L\to K$, called a 
\emph{section}. 
\par
\begin{prop}\label{pro:Astable}
If $\cA$ is a real *-subalgebra of $\M_n$, then the orthogonal 
projection $\cD(\M_n)\to\cD(\cA)$ is a retraction and 
$\cD(\cA)$ is stable.
\end{prop}
\begin{proof}
The orthogonal projection $\pi:\cD(\M_n)\to\cD(\cA)$ is well defined by 
Lemma~\ref{lem:dualcones} and it is a retraction because the inclusion 
$\cD(\cA)\subset\cD(\M_n)$ of equation \eqref{eq:D-inter} provides a 
section for $\pi$. Coro.~1.3 in \cite{Clausing1978} asserts that the 
image of a stable convex set under a retraction is stable. Therefore, 
and since $\cD(\M_n)$ is stable by Rem.~\ref{rem:Dstable}~b), the convex 
set $\cD(\cA)$ is stable.
\end{proof}
\begin{exa}
Prop.~\ref{pro:Astable} does not generalize to arbitrary matrix systems. 
We consider the hermitian matrices 
\[
A_1:=X\oplus 1\oplus 1\,,
\quad
A_2:=Z\oplus 0\oplus 0\,, 
\quad 
A_3:=0\oplus (-1)\oplus 1\,,
\]
in the algebra $\M_2\oplus\M_1\oplus\M_1$. The joint numerical range 
$V(A_1,A_2,A_3)$, introduced in Sec.~\ref{sec:numerical-ranges}, is easier 
to handle algebraically than the set of density matrices 
$\cD(\cR(A_1,A_2,A_3))$, to which it is affinely isomorphic by 
Lemma~\ref{lem:jnr}. Since $A_1,A_2,A_3\in\M_2\oplus\M_1\oplus\M_1$, we 
have the standard result of 
\[\textstyle
V(A_1,A_2,A_3)
=\conv\big( V(X,Z,0)
\cup V(1,0,-1)
\cup V(1,0,1) \big)\,.
\]
Hence, $V(A_1,A_2,A_3)=\conv(S)$ is the convex hull of 
\[\textstyle
S:=\{(c_1,c_2,0)\tp\in\R^3\mid c_1^2+c_2^2=1\}
\cup\{(1,0,-1)\tp,(1,0,1)\tp\}\,.
\]
The set of extreme points $S\setminus\{(1,0,0)\tp\}$ of $V(A_1,A_2,A_3)$ is 
not closed, hence $V(A_1,A_2,A_3)$ is not stable by Rem.~\ref{rem:Dstable}~a).
\end{exa}
%
%
\section{Proof of Thm.~\ref{thm:Aopen}}
\label{sec:proof-thm}
The main ideas in establishing Thm.~\ref{thm:Aopen} are that the set of density 
matrices $\cD(\M_n)$ is a stable convex set and that $\cD(\M_n)$ is highly 
symmetric. Being stable and symmetric, $\cD(\M_{p+q})$ projects openly onto 
$\cD(\M_p\oplus\M_q)$. Loosely speaking, Thm.~\ref{thm:Aopen} is obtained by 
combining various such open projections, since every *-subalgebra of $\M_n$ 
is a direct sum of full matrix algebras.
\par
If $K$ is a stable convex set, then the \emph{arithmetic mean map}
\[\textstyle
K^{\times k}\to K\,,
\quad
(x_1,\dots,x_k)\mapsto\frac{1}{k}\sum_{i=1}^k x_i\,,
\]
defined on the $k$-fold cartesian product of $K$, is open for all positive 
integers $k$. More generally, for every tuple $(s_1,\dots,s_k)$ of 
nonnegative real numbers adding up to one, the map
\[\textstyle
K^{\times k}\to K\,,
\quad
(x_1,\dots,x_k)\mapsto\sum_{i=1}^k s_i x_i
\]
is open. The latter assertion is proved for $k=2$ in Prop.~1.1 in 
\cite{ClausingPapadopoulou1978}. By induction, it is true for every 
$k>2$ as well.
\par
\begin{lem}\label{lem:cyclic-group}
Let $\cR$ be a real matrix system on $\C^n$ and let $\cD(\cR)$ be stable
and invariant under an orthogonal transformation $\gamma:\cR\to\cR$ that 
generates a finite cyclic group. Let $\cA:=\{A\in\cR:\gamma(A)=A\}$ be a 
real *-subalgebra of $\M_n$. Then the orthogonal projection 
$\cD(\cR)\to\cD(\cA)$ is open. 
\end{lem}
\begin{proof}
Let $k$ denote the order of the cyclic group generated by $\gamma$. The main idea 
is to write the orthogonal projection $\pi:\cD(\cR)\to\cD(\cA)$ in terms of the 
arithmetic mean map
\[\textstyle
a:\cD(\cR)^{\times k}\to\cD(\cR)\,,
\quad
(\rho_1,\dots,\rho_k)\mapsto\frac{1}{k}\sum_{i=1}^k\rho_i\,,
\]
which, as discussed above, is open because $\cD(\cR)$ is stable. The proof is done 
in two steps. First, we prove that 
\begin{align}\label{eq:group1}
a(\gamma(\cO)\times\dots\times\gamma^k(\cO))\cap\cA
\end{align}
is an open subset of $\cD(\cA)$ for all open subsets $\cO$ of 
$\cD(\cR)$. Secondly, we prove that 
\begin{align}\label{eq:group2}
a(\gamma(\cK)\times\dots\times\gamma^k(\cK))\cap\cA=\pi(\cK)
\end{align}
holds for all convex subsets $\cK$ of $\cD(\cR)$. The assertions \eqref{eq:group1}
and \eqref{eq:group2} together show that $\pi(\cO)$ is an open subset of $\cD(\cA)$ 
for all convex open subsets $\cO$ of $\cD(\cR)$, and hence for all open subsets.
\par
First, we prove that the set in \eqref{eq:group1} is open. As $\cD(\cR)$ is 
invariant under the orthogonal transformation $\gamma$, the map $\gamma$ 
restricts to a homeomorphism $\cD(\cR)\to\cD(\cR)$. It follows that 
$\gamma(\cO)\times\dots\times\gamma^k(\cO)$ is an open subset of the $k$-fold 
cartesian product $\cD(\cR)^{\times k}$. Then
\[
\widetilde{\cO}:=a(\gamma(\cO)\times\dots\times\gamma^k(\cO))
\]
is an open subset of $\cD(\cR)$, because $\cD(\cR)$ is stable. Finally, 
$\widetilde{\cO}\cap\cA$ is an open subset of $\cD(\cA)$ by equation 
\eqref{eq:D-inter}, which shows that $\widetilde{\cO}\cap\cA$ equals 
$\widetilde{\cO}\cap\cD(\cA)$. 
\par
Secondly, we prove the formula \eqref{eq:group2}, beginning with the inclusion 
``$\supset$''. Let $\rho\in\cK\subset\cD(\cR)$. The density matrix
$\sigma:=a(\gamma(\rho),\dots,\gamma^k(\rho))$ is invariant under $\gamma$. 
This implies $\sigma\in\cA$, hence $\sigma\in\cD(\cA)$ by \eqref{eq:D-inter}. 
Since $\gamma$ is self-adjoint, for all $A\in\cH(\cA)$ we have
\[\textstyle
\Re\langle\rho,A\rangle
=\Re\langle a(\rho,\dots,\rho),A\rangle
=\Re\langle a(\gamma(\rho),\dots,\gamma^k(\rho)),A\rangle
=\Re\langle\sigma,A\rangle\,,
\]
hence $\pi(\rho)=\sigma$. Regarding the inclusion ``$\subset$'',
let $\rho_i\in\cK$, $i=1,\dots,k$. Let 
$\sigma:=a(\gamma(\rho_1),\dots,\gamma^k(\rho_k))$ and 
$\tau:=a(\rho_1,\dots,\rho_k)$. Then for all $A\in\cH(\cA)$
\[\textstyle
\Re\langle\sigma,A\rangle
=\Re\langle a(\gamma(\rho_1),\dots,\gamma^k(\rho_k)),A\rangle
=\Re\langle a(\rho_1,\dots,\rho_k),A\rangle
=\Re\langle\tau,A\rangle
\]
holds. If $\sigma\in\cA$, then $\sigma=\pi(\tau)$ follows. 
As $\cK$ is convex, $\tau\in\cK$ holds and we obtain $\sigma\in\pi(\cK)$.
\end{proof}
\begin{rem}\label{rem:Schur}
A hermitian matrix $M\in\cH(M_{p+q})$ in the block form
\[
M=\begin{pmatrix}
A & B\\
B^\ast & C
\end{pmatrix}\,,
\quad
A\in\cH(M_p)\,,
\quad
B\in\C^{p\times q}\,,
\quad
C\in\cH(M_q)
\]
is positive semidefinite if and only if the top left block $A$ is positive
semidefinite, the range of $B$ is included in the range of $A$, and the 
\emph{generalized Schur complement} $M/A=C-B^\ast A^-B$ is positive 
semidefinite, where $A^-$ is a \emph{generalized inverse} of $A$, that is 
to say, $A^-\in M_p$ and $AA^-A=A$ holds \cite{HornZhang2005}.
\end{rem}
\begin{prop}\label{pro:project-sum}
The orthogonal projection $\cD(\M_{p+q})\to\cD(\M_p\oplus\M_q)$ is open.
\end{prop}
\begin{proof}
Lemma~\ref{lem:cyclic-group} proves the claim when $\cR:=\M_{p+q}$ and 
$\cA:=\M_p\oplus\M_q$. The set of density matrices $\cD(\M_{p+q})$ is stable 
by Rem.~\ref{rem:Dstable}~b). The reflection $\gamma:\M_{p+q}\to\M_{p+q}$ at 
the subspace $\M_p\oplus\M_q$ generates a group of order two. 
\par
We show that $\cD(\M_{p+q})$ is invariant under $\gamma$. In block form, 
the reflection reads
\[\textstyle
\gamma:
\begin{pmatrix}
A & B\\
D & C
\end{pmatrix}
\mapsto
\begin{pmatrix}
A & -B\\
-D & C
\end{pmatrix}\,,
\]
where $A\in\M_p$, $B\in\C^{p\times q}$, $C\in\M_q$, and $D\in\C^{q\times p}$.
The space of hermitian matrices is invariant under $\gamma$, which restricts to 
\[\textstyle
\cH(\M_{p+q})\to\cH(\M_{p+q})\,,
\quad
M=\begin{pmatrix}
A & B\\
B^\ast & C
\end{pmatrix}
\mapsto
M'=\begin{pmatrix}
A & -B\\
-B^\ast & C
\end{pmatrix}\,,
\]
where $A\in\cH(\M_p)$, $B\in\C^{p\times q}$, and $C\in\cH(\M_q)$. By
Rem.~\ref{rem:Schur}, the map $M\mapsto M'$ preserves the positive 
semidefiniteness since $M$ and $M'$ have the same diagonal blocks and 
both off-diagonal blocks differ in a sign, so that $M'/A=M/A$ holds. 
The map $\gamma$ preserves the trace. The set of fixed points
$\M_p\oplus\M_q=\{A\in\M_{p+q}:\gamma(A)=A\}$ is a *-subalgebra of $\M_{p+q}$.
\end{proof}
\begin{cor}\label{cor:project-sum}
The orthogonal projection 
$\cD(\M_n)\to\cD(\M_{n_1}\oplus\dots\oplus\M_{n_k})$ is open,
where $n=n_1+\dots+n_k$.
\end{cor}
\begin{proof}
Proceeding by induction, we observe that the orthogonal projection 
$\cD(\M_{n_1})\to\cD(\M_{n_1})$ is of course open. Let $k\geq 1$ and 
assume that the orthogonal projection
\[\textstyle
\pi_k:\cD(\M_{n_1+\dots+n_k})\to\cD(\M_{n_1}\oplus\dots\oplus\M_{n_k})
\]
is open. The orthogonal projection 
$\M_{n_1+\dots+n_{k+1}}\to\M_{n_1}\oplus\dots\oplus\M_{n_{k+1}}$ factors 
into the orthogonal projections
\begin{align*}
\M_{n_1+\dots+n_{k+1}}
&\to\M_{n_1+\dots+n_k}\oplus\M_{n_{k+1}}\\
\text{and}\quad
\M_{n_1+\dots+n_k}\oplus\M_{n_{k+1}}
&\to\M_{n_1}\oplus\dots\oplus\M_{n_k}\oplus\M_{n_{k+1}}\,.
\end{align*}
Hence, by Lemma~\ref{lem:dualcones}, the map $\pi_{k+1}$ factors into
the orthogonal projections
\begin{align}
\label{eq:corodcs1}
\cD(\M_{n_1+\dots+n_{k+1}})
&\to\cD(\M_{n_1+\dots+n_k}\oplus\M_{n_{k+1}})\\
\label{eq:corodcs2}
\text{and}\quad
\cD(\M_{n_1+\dots+n_k}\oplus\M_{n_{k+1}})
&\to\cD(\M_{n_1}\oplus\dots\oplus\M_{n_k}\oplus\M_{n_{k+1}})\,.
\end{align}
The map \eqref{eq:corodcs1} is open by Prop.~\ref{pro:project-sum}.
Lemma~\ref{lem:dir-sum-proj} shows that \eqref{eq:corodcs2} equals
\begin{align*}
\cD(\M_{n_1+\dots+n_k})\dcs\cD(\M_{n_{k+1}})
&\to\cD(\M_{n_1}\oplus\dots\oplus\M_{n_k})\dcs\cD(\M_{n_{k+1}})\,,
\end{align*}
which is open by Prop.~\ref{pro:dcs-map} (for $m=2$) and by the 
induction hypothesis. Being the composition of two open maps, 
$\pi_{k+1}$ is open.
\end{proof}
\begin{prop}\label{pro:project-tensor}
The orthogonal projection 
$\cD(\bigoplus_{i=1}^k\M_q)\to\cD(\M_q\otimes\id_k)$ is open.
\end{prop}
\begin{proof}
Lemma~\ref{lem:cyclic-group} proves the claim when $\cR:=\bigoplus_{i=1}^k\M_q$ 
and $\cA:=\M_q\otimes\id_k$. The convex set $\cD(\cR)$ is stable by 
Rem.~\ref{rem:Dstable}~c) and equals the $k$-fold direct convex sum
$\cD(\cR)=\cD(\M_q)\dcs\cdots\dcs\cD(\M_q)$ of $\cD(\M_q)$ by 
Lemma~\ref{lem:dir-sum}. Hence, the cyclic permutation $\gamma=(1,\dots,k)$ 
defines the orthogonal transformation
\[\textstyle
\gamma:\cD(\cR)\to\cD(\cR)\,,
\quad
(\sigma_1,\dots,\sigma_k)
\mapsto(\sigma_{\gamma^{-1}(1)},\dots,\sigma_{\gamma^{-1}(k)})\,,
\]
which generates a group of order $k$. Clearly, $\cD(\cR)$ is invariant under 
$\gamma$ and $\cA=\{A\in\cR:\gamma(A)=A\}$ is a *-subalgebra of $\M_{kq}$.
\end{proof}
\begin{proof}[Proof of Thm.~\ref{thm:Aopen}]
Since $\cA$ is a *-subalgebra of $\M_n$, there exists a unitary $n\times n$ matrix 
$U$ such that $U\cA U^\ast=\bigoplus_{i=1}^m\cA_i$, where 
$\cA_i:=\M_{q_i}\otimes\id_{k_i}$ for every $i=1,\dots,m$, and 
$q_1k_1+\dots+q_mk_m=n$, see \cite[Thm.~5.6]{Farenick2001}.
\par
As $\cD(\M_n)\to\cD(\M_n)$, $\rho\mapsto U\rho U^\ast$ is a homeomorphism, it 
suffices to prove that the orthogonal projection 
$\pi:\cD(\M_n)\to\cD(\bigoplus_{i=1}^m\cA_i)$ is open. The orthogonal projection 
$\M_n\to\bigoplus_{i=1}^m\cA_i$ factors into the orthogonal projections
\begin{align*}
\M_n & \textstyle
\to\cB_1\oplus\dots\oplus\cB_m\\
\text{and}\quad
\cB_1\oplus\dots\oplus\cB_m & \textstyle
\to\cA_1\oplus\dots\oplus\cA_m\,,
\end{align*}
where $\cB_i:=\bigoplus_{j=1}^{k_i}\M_{q_i}$, $i=1,\dots,m$. By 
Lemma~\ref{lem:dualcones}, the map $\pi$ factors into
\begin{align}
\label{eq:thmdcs1}
\cD(\M_n)
&\to\cD(\cB_1\oplus\dots\oplus\cB_m)\,,\\
\label{eq:thmdcs2}
\text{and}\quad
\cD(\cB_1\oplus\dots\oplus\cB_m)
&\to\cD(\cA_1\oplus\dots\oplus\cA_m)\,.
\end{align}
The map \eqref{eq:thmdcs1} is open by Coro.~\ref{cor:project-sum} 
(for $k=k_1+\dots+k_m$). Lemma~\ref{lem:dir-sum-proj} shows that the map
\eqref{eq:thmdcs2} equals
\begin{align*}
\cD(\cB_1)\dcs\cdots\dcs\cD(\cB_m)
&\to\cD(\cA_1)\dcs\cdots\dcs\cD(\cA_m)\,,
\end{align*}
which is open by Prop.~\ref{pro:dcs-map} and Prop.~\ref{pro:project-tensor}. 
In conclusion, $\pi$ is open as it is a composition of two open maps.
\end{proof}
%
%
%
\section{Real *-subalgebras of $\M_n$}
\label{sec:real}
Every real *-subalgebra of $\M_n$ is *-isomorphic \cite[Thm.~5.22]{Farenick2001} 
to a direct sum of algebras of real, complex, and quaternionic 
\mbox{$q$-by-$q$}-matrices of various sizes $q$. We are here interested in the 
algebra $\M_n(\R)$ of real $n\times n$ matrices.
\par
\begin{prop}\label{pro:project-real}
The orthogonal projection $\cD(\M_n)\to\cD(\M_n(\R))$ is open.
\end{prop}
\begin{proof}
Lemma~\ref{lem:cyclic-group} proves the claim when $\cR:=\M_n$ and 
$\cA:=\M_n(\R)$. The set of density matrices $\cD(\M_n)$ is stable 
by Rem.~\ref{rem:Dstable}~b). The reflection $\gamma:\M_n\to\M_n$ at 
the real subspace $\M_n(\R)$ generates a group of order two. 
\par
The algebra $\M_n$ 
is the orthogonal direct sum $\M_n=\M_n(\R)\oplus\ii\,\M_n(\R)$ and the 
reflection at the real subspace $\M_n(\R)$ reads
\[\textstyle
\gamma:
\M_n\to\M_n\,,
\quad
A+\ii B\mapsto A-\ii B\,, 
\quad
A,B\in\M_n(\R)\,.
\]  
The orthogonal transformation $\gamma$ preserves the space of hermitian matrices, 
which is the orthogonal direct sum
\[
\cH(\M_n)=\mathrm{Sym}_n(\R)\oplus\ii\,\mathrm{Skew}_n(\R)
\]
of the space of real symmetric matrices 
\begin{align*}
\mathrm{Sym}_n(\R) 
 & :=\{A\in\M_n(\R)\colon A\tp=A\}
 =\cH(\M_n(\R))
 \end{align*}
and the space of skew-symmetric matrices
\begin{align*}
\mathrm{Skew}_n(\R)
 &:=\{A\in\M_n(\R)\colon A\tp=-A\}\,.
\end{align*}
We prove that $\gamma$ preserves the trace and the positive semidefiniteness 
on the space of hermitian matrices. Let $A,C\in\mathrm{Sym}_n(\R)$ and 
$B,D\in\mathrm{Skew}_n(\R)$. Then 
$\tr(A+\ii B)=\tr(A)$ shows that the trace is preserved. It is well known 
that a matrix is positive semidefinite if and only if its inner product 
with the square of every hermitian matrix is nonnegative. Thus
\begin{align*}
\langle A-\ii B,(C+\ii D)^2\rangle
&=\langle A+\ii B,(C-\ii D)^2\rangle
\end{align*}
shows that $A-\ii B$ is positive semidefinite if $A+\ii B$ is positive 
semidefinite. Clearly, $\M_n(\R)=\{A\in\M_n:\gamma(A)=A\}$ is a real 
*-subalgebra of $\M_n$.
\end{proof}
As the orthogonal projection $\M_n\to\M_n(\R)$ is the entrywise real part,
we denote by $\Re:\cD(\M_n)\to\cD(\M_n(\R))$ the orthogonal projection 
of $\cD(\M_n)$ onto $\cD(\M_n(\R))$.
\par
\begin{exa}\label{ex:sum2real}
We consider the chain $\M_3\supset\M_2\oplus\M_1\supset\M_2(\R)\oplus\M_1(\R)$
of real *-subalgebras of $\M_3$. The orthogonal projections
\[\textstyle
\M_3
\longrightarrow
\M_2\oplus\M_1
\longrightarrow
\M_2(\R)\oplus\M_1(\R)
\]
restrict by Lemma~\ref{lem:dualcones} to 
\[\textstyle
\cD(\M_3)
\stackrel{\pi_1}{\longrightarrow}
\cD(\M_2\oplus\M_1)
\stackrel{\pi_2}{\longrightarrow}
\cD(\M_2(\R)\oplus\M_1(\R))\,.
\]
The map $\pi_1$ is open by Prop.~\ref{pro:project-sum}. By 
Lemma~\ref{lem:dir-sum-proj}, the map $\pi_2$ is the direct convex sum
\[\textstyle
\Re\dcs\idty:
\cD(\M_2)\dcs\{1\}
\to\cD(\M_2(\R))\dcs\{1\}
\]
of $\Re:\cD(\M_2)\to\cD(\M_2(\R))$ and the identity map $\idty:\{1\}\to\{1\}$. 
The map $\Re$ is open by Prop.~\ref{pro:project-real}, hence $\pi_2$ is open 
by Prop.~\ref{pro:dcs-map}. The orthogonal projection 
$\cD(\M_3)\to\cD(\M_2(\R)\oplus\M_1(\R))$ is open, as it is a composition of 
two open maps.
\end{exa}
%
%
\section{Topology simplified by algebra}
\label{sec:simpler-subalgebras}
Thm.~\ref{thm:Aopen} can simplify topology problems. Given topological spaces 
$K,L$, a map $f:K\to L$ is \emph{continuous} \cite{Kelley1975} at $x\in K$ if 
the preimage of every neighborhood of $f(x)$ in $L$ is a neighborhood of $x$ 
in $K$.
\par
\begin{lem}\label{lem:simplify}
Let $\cR_1,\cR_2$ be real matrix systems on $\C^n$ such that
\mbox{$\cR_2\subset\cR_1$}. Let $\pi_1:\cD(\M_n)\to\cD(\cR_1)$ and 
$\pi_2:\cD(\cR_1)\to\cD(\cR_2)$ denote the orthogonal projections and assume the 
orthogonal projection $\pi_2\circ\pi_1:\cD(\M_n)\to\cD(\cR_2)$ is open. Let 
$f:\cD(\cR_2)\to T$ be a map to a topological space $T$. Let $\rho\in\cD(\cR_1)$.
\begin{enumerate}
\item[a)] The map $f\circ\pi_2:\cD(\cR_1)\to T$ is open at $\rho$ if and only if
$f:\cD(\cR_2)\to T$ is open at $\pi_2(\rho)$.
\item[b)] The map $f\circ\pi_2:\cD(\cR_1)\to T$ is continuous at $\rho$ if and 
only if the map $f:\cD(\cR_2)\to T$ is continuous at $\pi_2(\rho)$.
\end{enumerate}
\end{lem}
\begin{proof}
The orthogonal projection $\cD(\M_n)\to\cD(\cR_2)$ equals indeed 
$\pi_2\circ\pi_1$ by Lemma~\ref{lem:dualcones}.
\par
We begin with the implication ``$\Rightarrow$'' of a). If $\cN_2\subset\cD(\cR_2)$ 
is a neighborhood of $\pi_2(\rho)$, then 
\[\textstyle
f(\cN_2)=(f\circ\pi_2)\circ\pi_2^{-1}(\cN_2)
\]
is a neighborhood of $f(\pi_2(\rho))$ because $\pi_2$ is continuous and 
$f\circ\pi_2$ is open at $\rho$. Regarding the implication ``$\Leftarrow$'', 
we choose a neighborhood $\cN_1\subset\cD(\cR_1)$ of $\rho$. Then
\[\textstyle
f\circ\pi_2(\cN_1)=f\circ(\pi_2\circ\pi_1)\circ\pi_1^{-1}(\cN_1)
\]
is a neighborhood of $f(\pi_2(\rho))$ because $\pi_1$ is continuous, 
$\pi_2\circ\pi_1$ is open, and $f$ is open at $\pi_2(\rho)$.
\par
To prove  b) we choose a neighborhood $\cN_T\subset T$ of $f(\pi_2(\rho))$. 
Regarding the implication ``$\Rightarrow$'', the preimage 
\[\textstyle
f^{-1}(\cN_T)=(\pi_2\circ\pi_1)\circ\pi_1^{-1}\circ(f\circ\pi_2)^{-1}(\cN_T)
\]
is a neighborhood of $\pi_2(\rho)$, because $f\circ\pi_2$ is continuous 
at $\rho$, the map $\pi_1$ is continuous, and $\pi_2\circ\pi_1$ is open.
Regarding the implication ``$\Leftarrow$'', the preimage
\[\textstyle
(f\circ\pi_2)^{-1}(\cN_T)=\pi_2^{-1}\circ f^{-1}(\cN_T)
\]
is a neighborhood of $\rho$, as $f$ is continuous at $\pi_2(\rho)$, 
and $\pi_2$ is continuous.
\end{proof}
\begin{rem}[Simplifying openness problems]\label{rem:simplifyA}
Let $\pi:\cD(\M_n)\to\cD(\cR)$ be the orthogonal projection to a real 
matrix system $\cR$ on $\C^n$ and let $\cA$ be a *-subalgebra of $\M_n$ 
such that $\cR\subset\cA$. Then $\pi=\pi_\cR\circ\pi_\cA$ factors into 
the orthogonal projections $\pi_\cA:\cD(\M_n)\to\cD(\cA)$ and 
$\pi_\cR:\cD(\cA)\to\cD(\cR)$. For every $\rho\in\cD(\M_n)$ the map 
$\pi$ is open at $\rho$ if and only if $\pi_\cR$ is open at 
$\pi_\cA(\rho)$.
\par
Indeed, the map $\pi$ factors by Lemma~\ref{lem:dualcones}. The second 
claim follows from Lemma~\ref{lem:simplify}~a), by letting $\cR_1:=\M_n$ 
and $\cR_2:=\cA$, and by taking $\pi_\cR$ as the map $f:\cD(\cR_2)\to T$. 
The assumptions of the lemma are met since $\pi_\cA$ is open by 
Thm.~\ref{thm:Aopen}.
\end{rem}
The following examples demonstrate the use of Rem.~\ref{rem:simplifyA}.
Continuity problems are a topic of Sec.~\ref{sec:quantum} below. 
\par
\begin{exa}\label{ex:ex1c}
Keeping the notation of $\pi=\pi_\cR\circ\pi_\cA$ from 
Rem.~\ref{rem:simplifyA}, we take $\cR:=\spn_\R(\id_3,X\oplus 1,Z\oplus 0)$
and $\cA:=\M_2\oplus\M_1$.
\begin{enumerate}
\item[a)]
By Ex.~\ref{ex:ex1b1}~a), the map $\pi_\cR$ is not open at any point in the 
half-open segment 
$\cG_0=\big(\ket{+}\!\!\bra{+}\oplus0,0\oplus 1\big]\subset\cD(\cA)$ and 
open at every point in the complement. By Rem.~\ref{rem:simplifyA}, the map 
$\pi:\cD(\M_3)\to\cD(\cR)$ is not open at any point in $\pi_\cA^{-1}(\cG_0)$ 
and open at every point in the complement.
\item[b)]
To describe $\pi_\cA^{-1}(\cG_0)$, we study the fibers of the orthogonal 
projection 
\[\textstyle
\cH(\M_3)\to\cH(\cA)\,,
\quad
\begin{pmatrix}
A & \ket{\varphi}\\
\bra{\varphi} & c
\end{pmatrix}
\mapsto
\begin{pmatrix}
A & 0\\
0 & c
\end{pmatrix}\,,
\]
where $A\in\cH(\M_2)$, $\ket{\varphi}\in\C^2\cong\C^{2\times 1}$, and 
$c\in\R\cong\cH(\M_1)$. Every point in 
$\cG=\big[\ket{+}\!\!\bra{+}\oplus0,0\oplus 1\big]$ is of the form 
$(1-\lambda)\ket{+}\!\!\bra{+}\oplus\lambda$ for some $\lambda\in[0,1]$. 
Using the generalized Schur complement (Rem.~\ref{rem:Schur}), one 
verifies that the fiber of $\pi_\cA$ over this point is the set of all 
matrices
\[\textstyle
\rho(\lambda,z):=
\begin{pmatrix}
(1-\lambda)\ket{+}\!\!\bra{+} & z\ket{+}\\
\overline{z}\bra{+} & \lambda
\end{pmatrix}\,,
\quad
z\in\C\,,
\quad
|z|^2\leq\lambda(1-\lambda)\,,
\]
where $|z|$ denotes the absolute value of $z\in\C$. In conclusion, the 
orthogonal projection $\pi:\cD(\M_3)\to\cD(\cR)$ is not open at any point 
of
\[\textstyle
\pi_\cA^{-1}(\cG_0)
=\{\rho(\lambda,z)\colon 
z\in\C, |z|^2\leq\lambda(1-\lambda),
\lambda\in(0,1]\}
\]
and open at every point in the complement.
\item[c)]
We verify a claim made in Ex.~\ref{ex:riM2M1}. The segment 
$\cG=\big[\ket{+}\!\!\bra{+}\oplus0,0\oplus 1\big]$ is the fiber of 
$\pi_\cR$ over $M=\frac{1}{2}(\ket{+}\!\!\bra{+}\oplus 1)\in\cD(\cR)$ by
Ex.~\ref{ex:ex1b1}~b), so $\pi_\cA^{-1}(\cG)$ is the fiber of 
$\pi=\pi_\cR\circ\pi_\cA$ over $M$. As recalled in part a) above, the map 
$\pi$ is open at $\ket{+}\!\!\bra{+}\oplus0$ but not open at any point of 
$\pi_\cA^{-1}(\cG_0)$. We now observe that $\pi$ is not open at any point 
in the relative interior of $\pi^{-1}(M)$, as we have the chain of 
inclusions
\begin{align*}
&\textstyle 
\hphantom{{}={}}
\pi_\cA\Big(\ri\big((\pi_\cR\circ\pi_\cA)^{-1}(M)\big)\Big)
=\ri\Big(\pi_\cA\big((\pi_\cR\circ\pi_\cA)^{-1}(M)\big)\Big)\\
&\textstyle 
=\ri(\pi_\cR^{-1}(M))
=\ri(\cG)
=\cG_0\setminus\{0\oplus 1\}
\subset\cG_0\,,
\end{align*}
whose first equality holds by \cite[Thm.~6.6]{Rockafellar1970}.
\end{enumerate}
\end{exa}
\begin{exa}\label{ex:2proj}
The recipe of Rem.~\ref{rem:simplifyA} helps analyze the openness of
the orthogonal projection $\cD(\M_n)\to\cD(\cR(P,Q))$ to the real matrix system 
\[\textstyle
\cR(P,Q)=\spn_\R(\id_n,P,Q)
\]
generated by two orthogonal projections $P,Q\in\M_n$, that is to say, matrices
satisfying $P=P^2=P^\ast$ and $Q=Q^2=Q^\ast$. It is well known that the matrix 
system $\cR(P,Q)$ is included in a surprisingly small *-subalgebra of $\M_n$, 
see Coro.~2.2 in the survey \cite{BoettcherSpitkovsky2010} by Böttcher and 
Spitkovsky, and the references therein. More precisely, there exists a unitary 
$n\times n$ matrix $U$, nonnegative integers $m_1\leq 4$ and $m_2$, and positive 
integers $k_i$, $i=1,\ldots,m_1$ satisfying $k_1+\dots+k_{m_1}+2m_2=n$, such 
that $\cR:=U\cR(P,Q)U^\ast$ is included in 
\[\textstyle
\cA:=\big(\bigoplus_{i=1}^{m_1}\M_1\otimes\id_{k_i}\big)
\oplus
\big(\bigoplus_{j=1}^{m_2}\M_2\big)\,.
\]
Since $\rho\mapsto U\rho U^\ast$ is a homeomorphism of $\cD(\M_n)$, the 
openness problems of the orthogonal projections $\cD(\M_n)\to\cD(\cR(P,Q))$ 
and $\cD(\M_n)\to\cD(\cR)$ are equivalent. The second one is substantially 
simplified by the method of Rem.~\ref{rem:simplifyA} as $\cD(\cA)$ has a 
rather simple shape. It is the direct convex sum of several three-dimensional 
Euclidean balls and a simplex of dimension at most three by
Lemma~\ref{lem:dir-sum} and Ex.~\ref{ex:Bloch}. This observation should also 
simplify the strong continuity problem for the numerical range $W(P+\ii Q)$.
\end{exa}
%
%
\section{Continuity in quantum information theory}
\label{sec:quantum}
We discuss continuity problems of entropic inference maps and of measures 
of correlation. We assume that $\cR$ is a real matrix system on $\C^n$ and 
that, without loss of generality (see Rem.~\ref{rem:realvscomplex}), we 
have $\cR\subset\cH(\M_n)$. 
\par
\begin{exa}[Maximum entropy inference I]\label{exa:ME1}
The purpose of the maximum entropy inference method is to update a prior 
probability distribution if new information becomes available in the form of  
constraints that specify a set of possible posterior probability 
distributions. The preferred posterior is that which minimizes the relative 
entropy from the prior subject to the available constraints, see Chap.~8 in 
Caticha's book \cite{Caticha2022} and the references therein. An analogous 
quantum mechanical inference method can be defined by replacing probability 
distributions with density matrices and the standard relative entropy with 
the Umegaki relative entropy. The axiomatic foundations of the maximum 
entropy inference method were settled for probability distributions in the 
1980's, see Chap.~6 in \cite[pp.~157--160]{Caticha2022}. More than 30 years
later, the axioms of the quantum inference are still a matter of discussion
\cite{Ali-etal2012} but a new approach appeared in the work of Vanslette 
recently \cite{Vanslette2017}.
\par
Linear constraints%
\footnote{Linear constraints can be defined in terms of expectation values. 
Let $A_1,\dots,A_k$ be hermitian $n\times n$ matrices such that 
$\cR=\spn_\R(\id_n,A_1,\dots,A_k)$. The observables represented by 
$A_1,\dots,A_k$ have the \emph{expectation values}
$v(\rho)=(\langle\rho,A_1\rangle,\dots,\langle\rho,A_k\rangle)\tp$ if 
$\rho\in\cD(\M_n)$ is the system state \cite{BengtssonZyczkowski2017}. The 
fiber $\pi^{-1}(\sigma)$ over $\sigma\in\cD(\cR)$ is the set of 
$\rho\in\cD(\M_n)$ whose expectation values are $v(\rho)=v(\sigma)$, see
Lemma~\ref{lem:jnr}.}
on $\cD(\M_n)$ are defined by the orthogonal projection
\[\textstyle
\pi:\cD(\M_n)\to\cD(\cR)\,.
\]
The \emph{relative entropy} $S:\cD(\M_n)\times\cD(\M_n)\to[0,+\infty]$ 
is an asymmetric distance. It is defined by 
$S(\rho_1,\rho_2):=\tr[\rho_1(\log(\rho_1)-\log(\rho_2))]$ if the range 
of $\rho_1$ is included in the range of $\rho_2$ and by 
$S(\rho_1,\rho_2):=+\infty$ otherwise, for all $\rho_1,\rho_2\in\cD(\M_n)$. 
Let $\tau\in\cD(\M_n)$, the \emph{prior}, be a density matrix of 
maximal rank $n$. Then
\[\textstyle
\phi_\tau:\cD(\M_n)\to\R\,,
\quad
\rho\mapsto-S(\rho,\tau)
\]
is continuous and strictly concave. So the \emph{maximum entropy inference map}
\begin{equation}\label{eq:maxentinf}\textstyle
\Psi_\tau:\cD(\cR)\to\cD(\M_n)\,,
\quad
\sigma\mapsto\argmax_{\rho\in\pi^{-1}(\sigma)}\phi_\tau(\rho)
\end{equation}
is well defined, see \cite[Def.~1.1]{Weis2014} and the references therein. 
Discontinuities of this inference map \cite{WeisKnauf2012} aroused interest in 
theoretical physics 
\cite{Chen-etal2015,HuberGuehne2016,Lostaglio-etal2017,Winter2016}. The map 
\eqref{eq:maxentinf} is continuous for commutative real matrix systems $\cR$, 
for example for the inference of probability distributions mentioned above.
\end{exa}
We quote \cite[Thm.~4.9]{Weis2014}.
\par
\begin{thm}\label{thm:MEcont}
Let $\sigma\in\cD(\cR)$. Then $\Psi_\tau$ is continuous at $\sigma$ if and 
only if $\pi$ is open at $\Psi_\tau(\sigma)$.
\end{thm}
In what follows, continuity and openness problems will be simplified by 
factoring $\pi=\pi_\cR\circ\pi_\cA$ through a *-subalgebra $\cA$ of $\M_n$ 
that includes $\cR$, using the notation of Rem.~\ref{rem:simplifyA}.
\par
\begin{exa}[Maximum entropy inference II]\label{exa:ME2}
Before simplifying the openness problem of $\pi$, we describe the set 
$\Psi_\tau(\cD(\cR))$ of posteriors, as the openness only matters for
points in this set by Thm.~\ref{thm:MEcont}. Recalling $\cR=\cH(\cR)$, 
we define 
\[\textstyle
\cE_\tau(\cR)
:=\big\{ \frac{e^{\log(\tau)+A}}{\tr(e^{\log(\tau)+A})} \mid A\in\cR\big\}\,.
\]
The manifold $\cE_\tau(\cR)$ is known as a \emph{Gibbsian family} or 
\emph{exponential family}, see 
\cite{WeisKnauf2012,Niekamp-etal2013,Weis2014,Pavlov-etal2024} and the 
references therein. By (D5) in \cite{Weis2014} we have
\begin{equation}\label{eq:rIexpfamily}\textstyle
\Psi_\tau(\cD(\cR))=
\{\rho_1\in\cD(\M_n):\inf_{\rho_2\in\cE_\tau(\cR)}S(\rho_1,\rho_2)=0\}\,.
\end{equation}
The right-hand side of \eqref{eq:rIexpfamily} is the 
\emph{reverse information closure} or \emph{rI-closure} \cite{CsiszarMatus2003} 
of $\cE_\tau(\cR)$, which is a subset of the Euclidean closure of 
$\cE_\tau(\cR)$.
\par
\begin{enumerate}
\item[a)]
If $\tau=\id_n/n$ is the uniform prior, then
\[\textstyle
\phi_\tau(\rho)
=-S(\rho,\id_n/n)
=S(\rho)-\log(n)\,,
\quad
\rho\in\cD(\M_n)
\]
is the \emph{von Neumann entropy} $S(\rho):=-\tr[\rho\log(\rho)]$ up to a 
constant. By functional calculus, $\cE_\tau(\cR)$ is included in $\cA$ and 
so is the set of posteriors $\Psi_\tau(\cD(\cR))$ as per \eqref{eq:rIexpfamily}, 
because $\cA$ is closed. Rem.~\ref{rem:simplifyA} then shows that for every 
$\rho\in\Psi_\tau(\cD(\cR))$ the openness of $\pi$ at $\rho$ is equivalent 
to the openness of $\pi_\cR=\pi|_{\cD(\cA)}$ at $\rho$
(the same conclusion is true for every prior $\tau$ in $\cA$). This can 
simplify the problem if $\cA$ has a simpler structure or a smaller dimension 
than $\M_n$. An example is given in Ex.~\ref{ex:2proj} above.  
\item[b)]
If the prior $\tau$ lies outside of $\cA$ then $\cE_\tau(\cR)$ is disjoint 
from $\cD(\cA)$, again by functional calculus. Rem.~\ref{rem:simplifyA} 
then shows that for every $\rho\in\Psi_\tau(\cD(\cR))$ the openness of $\pi$ 
at $\rho$ is equivalent to the openness of $\pi_\cR$ at $\pi_\cA(\rho)$.
This is an even greater simplification than in part a) above, because the 
analysis of $\pi$ on $\cD(\M_n)$ is reduced to that of $\pi_\cR$ on 
$\cD(\cA)$.
\item[c)]
Independently of $\cA$, it sometimes helps that the posterior 
$\Psi_\tau(\sigma)$ is contained in the relative interior of the fiber 
$\pi^{-1}(\sigma)$ over $\sigma$ for every $\sigma\in\cD(\cR)$ and prior 
$\tau$ by Coro.~5.7 and Lemma~5.8 in \cite{Weis2014}. As an example, the 
orthogonal projection $\pi:\cD(\M_3)\to\cD(\cR)$ to the real matrix 
system $\cR$ in Ex.~\ref{ex:ex1c} is not open anywhere in the relative 
interior of the fiber over a certain point $M$ and open at every point 
in the complement of that fiber. Thm.~\ref{thm:MEcont} then shows that 
$\Psi_\tau$ is discontinuous at $M$ and continuous everywhere else in 
the ellipse $\cD(\cR)$ for every prior $\tau$.
\end{enumerate}
\end{exa}
From here on, we assume the prior $\tau:=\id_n/n$ be uniform. Using the 
von Neumann entropy $S=\phi_\tau+\log(n)$, we write the inference map 
\eqref{eq:maxentinf} from Ex.~\ref{exa:ME1} as 
\[\textstyle
\Psi:\cD(\cR)\to\cD(\M_n)\,,
\quad
\sigma\mapsto\argmax_{\rho\in\pi^{-1}(\sigma)}S(\rho)\,.
\]
We also write $\cE(\cR):=\cE_\tau(\cR)$ for the exponential family of 
Ex.~\ref{exa:ME2}. 
\par
\begin{exa}[Maximum entropy inference III]\label{exa:ME3}
An important example from physics is the real matrix system of 
\emph{local Hamiltonians} \cite{Chen-etal2015,Zeng-etal2019}.
\par
Every unit $i\in\Omega:=\{1,2,\ldots,N\}$ of an $N$-qubit system is associated 
with a copy $\cA_i$ of the algebra $\M_2$. The subsystem with units in a subset 
$\nu\subset\Omega$ is associated with the tensor product algebra 
$\cA_\nu:=\bigotimes_{i\in \nu}\cA_i$, whose identity we denote by $\id_\nu$. 
We have $\cA_\Omega=\M_n$ for $n=2^N$. The algebra $\cA_\nu$ embeds into 
$\cA_\Omega$ \emph{via} the map $\cA_\nu\to\cA_\Omega$,
$A\mapsto A\otimes\id_{\bar\nu}$, where $\bar\nu=\Omega\setminus\nu$ is the 
complement of $\nu$. Let $\fg$ be a family of subsets of $\Omega$. 
A \emph{$\fg$-local Hamiltonian} is a hermitian matrix in $\cA_\Omega$ of the 
form 
\[\textstyle
\sum_{\nu\in\fg}A_\nu\otimes\id_{\bar\nu}\,, 
\quad 
A_\nu\in\cH(\cA_\nu)\,,
\quad 
\nu\in\fg\,.
\]
We denote the real matrix system of all $\fg$-local Hamiltonians by $\cR_\fg$
and the orthogonal projection by $\pi_\fg:\cD(\cA_\Omega)\to\cD(\cR_\fg)$.
\par
The \emph{partial trace} $\tr_{\bar\nu}:\cA_\Omega\to\cA_\nu$ is the adjoint of 
the embedding $\cA_\nu\to\cA_\Omega$ and satisfies 
$\langle A\otimes\id_{\bar\nu},B\rangle=\langle A,\tr_{\bar\nu}(B)\rangle$ for 
every $A\in\cA_\nu$, $B\in\cA_\Omega$. The partial trace $\tr_{\bar\nu}(\rho)$ 
of $\rho\in\cD(\cA_\Omega)$ is a density matrix of $\cA_\nu$ called 
\emph{reduced density matrix}. Let
\[\textstyle
\red_\fg: 
\cA_\Omega\to\prod_{\nu\in\fg}\cA_\nu\,,
\quad
A\mapsto[\tr_{\bar\nu}(A)]_{\nu\in\fg}
\]
denote the map from $\cA_\Omega$ to the cartesian product of the algebras 
$(\cA_\nu)_{\nu\in\fg}$ that assigns reduced density matrices. 
\par
Linear constraints on $\cD(\cA_\Omega)$ have been defined in terms of reduced 
density matrices, see \cite{Niekamp-etal2013,Chen-etal2015} and
\cite[Sec.~1.4.2]{Zeng-etal2019}. This is formalized in the following diagram,
which commutes by formula (19) in \cite{WeisGouveia2023}. (Obvious restrictions 
of the domain and codomain of $\red_\fg$ are omitted in the sequel.)
\[\textstyle
\xymatrix{%
 & \cD(\cA_\Omega) \ar[d]^{\pi_\fg} \ar@<-1.0ex>[dl]_{\red_\fg} \\
 *+[l]{\red_\fg[\cD(\cA_\Omega)]} \ar@<-.5ex>[r]
 & \cD(\cR_\fg) \ar@<-.5ex>[l]_{\red_\fg}
}\]
The fiber $\pi_\fg^{-1}(\sigma)$ over $\sigma\in\cD(\cR_\fg)$ is the set of 
all $\rho\in\cD(\cA_\Omega)$ such that $\red_\fg(\rho)=\red_\fg(\sigma)$. 
Thm.~\ref{thm:MEcont} proves that the pullback $\Xi:=\Psi\circ\red_\fg^{-1}$ 
of the inference map $\Psi$ under $\red_\fg^{-1}$ is continuous at 
$(\rho_\nu)_{\nu\in\fg}\in\red_\fg[\cD(\cA_\Omega)]$ if and only if 
$\red_\fg$ is open at $\Xi[(\rho_\nu)_{\nu\in\fg}]\in\cD(\cA_\Omega)$.
\par
Chen et al.~\cite[Ex.~4]{Chen-etal2015} discovered a discontinuity of $\Xi$ 
for $N=3$ qubits and $\fg:=\{\{1,2\},\{2,3\},\{3,1\}\}$ at 
$(\rho_\nu)_{\nu\in\fg}\in\red_\fg[\cD(\cA_\Omega)]$, where 
\[\textstyle
\rho_\nu:=\frac{1}{2}(\ket{00}\!\!\bra{00}+\ket{11}\!\!\bra{11})\,,
\quad 
\nu\in\fg
\]
and they offered an interesting interpretation in terms of phase transitions. 
The map $\red_\fg$ being open\footnote{%
The openness of $\red_\fg$ at $\rho\in\cD(\cA_\Omega)$ is \emph{a priori} 
weaker than the continuity of $\Xi$ at $\red_\fg(\rho)$. The 
continuity means that any sufficiently small change of $\red_\fg(\rho)$ 
can be matched by an arbitrarily small change of $\rho$ \emph{inside} the 
image of $\Xi$ and not just anywhere in $\cD(\cA_\Omega)$. 
Somewhat surprisingly, the two propositions are equivalent by 
Thm.~\ref{thm:MEcont}.}
at $\rho\in\cD(\cA_\Omega)$ means that any sufficiently small change of 
$\red_\fg(\rho)$ in $\red_\fg[\cD(\cA_\Omega)]$ is matched by an 
arbitrarily small change of $\rho$ within $\cD(\cA_\Omega)$. Conversely, 
if the openness fails, then there are arbitrarily small changes of 
$\red_\fg(\rho)$ that can only be matched by changes of $\rho$ beyond 
some strictly positive threshold (in the metric sense). Loosely speaking, 
a small change of a subsystem abruptly changes the entire system. Such 
behavior is associated with phase transitions. It motivates every attempt 
to study the openness of $\pi_\fg$. This should be done by tuning the 
interaction pattern $\fg$ to a concrete system. Whether enclosing 
$\cR_\fg$ into a *-subalgebra of $\cA_\Omega$ could simplify this 
problem, as suggested by Exa.~\ref{exa:ME2}, is not yet clarified.
\end{exa}
We finish this paper with a map whose continuity is more subtle than 
that of the inference.
\par
\begin{exa}[Entropy distance I]\label{exa:dE1}%
The \emph{entropy distance} from the exponential family $\cE(\cR)$ is 
defined by
\begin{equation}\label{eq:dE}\textstyle
d:\cD(\M_n)\to\R\,,
\quad
\rho_1\mapsto\inf_{\rho_2\in\cE(\cR)}S(\rho_1,\rho_2)
\end{equation}
and equals the difference
\begin{equation}\label{eq:dEdiff}\textstyle
d(\rho)=S(\Psi\circ\pi(\rho))-S(\rho)\,,
\quad
\rho\in\cD(\M_n)
\end{equation}
between the value of the von Neumann entropy at $\rho$ and the maximal value 
on the fiber of $\pi$ that contains $\rho$, see p.~1288 in \cite{Weis2014}. 
Formula \eqref{eq:dEdiff} suggests studying the continuity of $d$ through 
the \emph{rI-projection} 
\[\textstyle
\Pi:\cD(\M_n)\to\cD(\M_n)\,,
\quad
\Pi:=\Psi\circ\pi\,.
\]
As per Def.~5.2 and equation (D8) in \cite{Weis2014}, the density matrix 
$\Pi(\rho_1)$ is the generalized rI-projection of 
$\rho_1\in\cD(\M_n)$ to $\cE(\cR)$, which is a well-known concept in 
probability theory \cite{CsiszarMatus2003}, and which is defined as follows. 
A sequence $(\tau_i)\subset\cD(\M_n)$ \emph{rI-converges} to 
$\rho_2\in\cD(\M_n)$ if $\lim_iS(\rho_2,\tau_i)=0$ holds. If every sequence 
$(\tau_i)\subset\cE(\cR)$ satisfying $\lim_iS(\rho_1,\tau_i)=d(\rho_1)$ 
rI-converges, independently of the sequence, to a unique $\rho_2\in\cD(\M_n)$, 
not necessarily in $\cE(\cR)$, then $\rho_2$ is the 
\emph{generalized rI-projection} of $\rho_1$ to $\cE(\cR)$.
\end{exa}
We quote from Lemma~5.15 and Lemma~4.5 in \cite{Weis2014}.
\par
\begin{lem}\label{lem:dEcont}~
\begin{enumerate}
\item[a)] For every $\rho\in\cD(\M_n)$ the rI-projection $\Pi$ is continuous 
at $\rho$ if and only if the entropy distance $d$ is continuous at $\rho$.
\item[b)] For every $\sigma\in\cD(\cR)$, the inference map $\Psi$ is continuous 
at $\sigma$ if and only if $d$ is continuous at every point in the fiber 
$\pi^{-1}(\sigma)$.
\end{enumerate}
\end{lem}
\begin{exa}[Entropy distance II]\label{exa:dE2}%
The entropy distance from the exponential family $\cE(\cR_\fg)$ of local 
Hamiltonians (Ex.~\ref{exa:ME3}) is interesting because it quantifies many-body 
correlations. Amari \cite{Amari2001} and Ay \cite{Ay2002} studied this type of 
correlation measures in probability theory. Linden et al.~\cite{Linden-etal2002} 
introduced it to quantum mechanics as a difference of von Neumann entropies like 
formula \eqref{eq:dEdiff}, see also \cite[Sec.~1.4.2]{Zeng-etal2019}. Zhou 
\cite{Zhou2009} proved the equality of the two representations \eqref{eq:dE} 
and \eqref{eq:dEdiff} for density matrices of maximal rank $n$, see also 
\cite{Niekamp-etal2013}; the equality is true without rank restrictions as 
stated in Ex.~\ref{exa:dE1}.
\end{exa}
\begin{prop}\label{pro:simp-d}
For every $\rho\in\cD(\M_n)$, the entropy distance $d$ is continuous at $\rho$
if and only if its restriction $d|_{\cD(\cA)}$ is continuous at $\pi_\cA(\rho)$.
\end{prop}
\begin{proof}
Let $\rho\in\cD(\M_n)$. By Lemma~\ref{lem:dEcont}~a), the entropy distance $d$ 
is continuous at $\rho$ if and only if the rI-projection $\Pi$ is continuous at 
$\rho$. The inference map
\[\textstyle
\Psi^\cA:\cD(\cR)\to\cD(\cA)\,,
\quad
\sigma\mapsto\argmax_{\eta\in\pi_\cR^{-1}(\sigma)}S(\eta)
\]
has the same values as $\Psi$, whose image $\Psi(\cD(\cR))$ is included in 
$\cD(\cA)$ by Ex.~\ref{exa:ME2}~a). Therefore, 
\[
\Pi=\Psi\circ\pi=\Psi\circ\pi_\cR\circ\pi_\cA
\]
is continuous at $\rho$ if and only if $\Psi^\cA\circ\pi_\cR\circ\pi_\cA$ 
is continuous at $\rho$, if and only if $\Pi^\cA\circ\pi_\cA$ is 
continuous at $\rho$, where 
\[
\Pi^\cA:\cD(\cA)\to\cD(\cA)\,,
\quad
\Pi^\cA:=\Psi^\cA\circ\pi_\cR\,.
\]
Since $\pi_\cA$ is open by Thm.~\ref{thm:Aopen}, it follows from 
Lemma~\ref{lem:simplify}~b) that $\Pi^\cA\circ\pi_\cA$ is continuous at 
$\rho$ if and only if $\Pi^\cA$ is continuous at $\pi_\cA(\rho)$. By
equation \eqref{eq:dEdiff},  
\[
d|_{\cD(\cA)}(\eta)
=S(\Psi^\cA\circ\pi_\cR(\eta))-S(\eta)\,,
\quad
\eta\in\cD(\cA)
\]
holds. Hence, $\Pi^\cA$ is continuous at $\pi_\cA(\rho)$ if and only if 
$d|_{\cD(\cA)}$ is continuous at $\pi_\cA(\rho)$, again by Lemma~5.15~1)
in \cite{Weis2014}.
\end{proof}
\begin{exa}[Entropy distance III]\label{exa:dE3}%
Prop.~\ref{pro:simp-d} helps solve the continuity problem of the entropy 
distance from $\cE(\cR)$ for $\cR:=\spn_\R(\id_3,X\oplus 1,Z\oplus 0)$ using 
the solution of the continuity problem of the restriction $d|_{\cD(\cA)}$ to 
the real *-subalgebra $\cA:=\M_2(\R)\oplus\M_1(\R)$ of $\M_3$. This solution 
was obtained from asymptotic curvature estimates \cite{Weis2014}. Real 
*-subalgebras are excluded from Proposition~\ref{pro:simp-d} but the conclusion 
is still true, as Ex.~\ref{ex:sum2real} can replace Thm.~\ref{thm:Aopen} in 
the proof of Prop.~\ref{pro:simp-d}.
\par
We recall from Ex.~\ref{ex:ex1b2} that the generatrix 
$\cG=[\ket{+}\!\!\bra{+}\oplus0,0\oplus 1]$ of the cone $\cD(\cA)$
is the fiber of the orthogonal projection $\pi_\cR:\cD(\cA)\to\cD(\cR)$
over $M:=\frac{1}{2}(\ket{+}\!\!\bra{+}\oplus 1)$. Thm.~5.18 in 
\cite{Weis2014} shows that $d|_{\cD(\cA)}$ is discontinuous at every point
in the half-open segment 
\[\textstyle
\cG_d:=
\big[\ket{+}\!\!\bra{+}\oplus0,M\big)
=\big\{(1-\lambda)\ket{+}\!\!\bra{+}\oplus\lambda
 \colon \lambda\in[0,\frac{1}{2}) \big\}
\]
and continuous at every point in the complement\footnote{%
Unitary similarity with respect to $U:=\exp(\ii \frac{\pi}{3\sqrt{3}}(X+Y+Z))$
permutes the Pauli matrices cyclicly. With respect to $U\oplus 1$ it matches 
$\cR$ and the problem considered in \cite{Weis2014}.}.
Prop.~\ref{pro:simp-d} then shows that the entropy distance 
$d$ is discontinuous at every point in $\pi_\cA^{-1}(\cG_d)$ and 
continuous at every point in the complement. 
\par
This result is consistent with the assertion of Ex.~\ref{exa:ME2}~c) that
the inference map $\Psi$ is discontinuous at $M$ and continuous anywhere 
else in the ellipse $\cD(\cR)$, as required by Lemma~\ref{lem:dEcont}~b). 
The points in the set $\pi_\cA^{-1}(\cG_d)$ are explicitly described in 
Ex.~\ref{ex:ex1c}~b). 
\end{exa}
%
%
\vskip\baselineskip\noindent
\textbf{Acknowledgments.} 
The author is grateful to
Ariel Caticha,
Achim Clausing, 
Chi-Kwong Li, 
Tim Netzer,
Ilya M.~Spitkovsky,
and Ivan Todorov
for remarks, feedback, and discussions. 
The author gave a lecture on the results of this paper in the special 
session ``Numerical Ranges'' of the conference IWOTA 2024 in 
Canterbury, UK, and wishes to thank the organizers for this 
opportunity.
\par
%
%
\bibliographystyle{plain}

%
%
%
\vspace{\baselineskip}
\parbox{10cm}{%
Stephan Weis\\
Czech Technical University in Prague\\ 
Faculty of Electrical Engineering\\
Karlovo nám\v{e}stí 13\\
12000, Prague 2\\
Czech Republic\\
e-mail \texttt{maths@weis-stephan.de}}
\end{document}